\documentclass[hidelinks,onefignum,onetabnum]{siamart220329}

\usepackage{amsfonts,amssymb}
\usepackage{enumitem}
\usepackage{bm}%
\setlist[enumerate,1]{label=(\roman*)}

\usepackage{mathtools}

\let\twoheadrightarrow\rrightarrow

\newcommand{\mapsmto}{\mathrel{\mathrlap{\mapsto}\mkern1mu\rightarrow}}

\DeclareMathOperator*{\dive}{div}
\DeclareMathOperator*{\esssup}{ess\,sup}

\newcommand{\ext}{\operatorname{ext}}
\newcommand{\diam}{\operatorname{diam}}
\newcommand{\clconv}{\overline{\operatorname{conv}}}

\newsiamremark{remark}{Remark}
\newsiamthm{claim}{Claim}
\newsiamthm{assumption}{Assumption}

\headers{Stability and generecity of bang-bang controls}{A. Dom\'inguez Corella and G. Wachsmuth}

\title{Stability and genericity of bang-bang controls in affine problems\thanks{
		\funding{Alberto is supported by the Alexander von Humboldt
			Foundation with a research fellowship. This work was partially supported by the Austrian Science Fund by means of the grant FWF I4571-N.}}}

\author{Alberto Dom\'inguez Corella\thanks{Friedrich-Alexander-Universit\"{a}t Erlangen-N\"{u}rnberg, Department of Mathematics, Chair for Dynamics, Control, Machine Learning and Numerics – Alexander von Humboldt Professorship, 91058 Erlangen, Germany; \tt \email{alberto.corella@fau.de, alberto.of.sonora@gmail.com}}
	\and Gerd Wachsmuth\thanks{Brandenburgische Technische Universit\"{a}t Cottbus–Senftenberg, Institute of Mathematics, 03046 Cottbus, Germany; \tt \email{wachsmuth@b-tu.de}}}

\usepackage{amsopn}

\begin{document}
	
\maketitle

\begin{abstract}
We analyse the role of the bang-bang property in affine optimal control problems.
We show that many essential stability properties of affine problems are only satisfied when  minimizers are bang-bang.
Moreover, we prove that almost any perturbation in an affine optimal control problem leads to a  bang-bang  strict global minimizer.
We work in an abstract framework that allows to cover many problems in the literature of optimal control, this includes problems constrained by partial and ordinary differential equations.
We give examples that show the applicability of our results to specific optimal control problems. 
\end{abstract}



\section{Introduction}
The term \textit{bang-bang} was coined a long time ago in control theory.  The term is informal and has become widely adopted in the field to refer to   a control  that switches from one extreme value to another; in analogy with relays, that make a \textit{bang}  to change from \textit{off} to \textit{on}, and another \textit{bang} to come back from \textit{on} to \textit{off}. The term also  applies  to controls that take  several values (more than two), usually the vertices of some polygon or polyhedron.

One of the fundamental principles in the theory of mathematical control is the so-called \textit{bang-bang principle}, this result asserts that for control systems, originating from an ordinary differential equation, any state that can be reached by a feasible control can also be attained by a bang-bang control, see, e.g., \cite{LaSalle,Neustadt}.
Bang-bang controls also arise in optimal control problems,  specially in \textit{affine problems}, where the control appears linearly (and hence the name). This is mainly because they lack the so-called Tykhonov regularization term, and hence some regularity of minimizers is lost, or better said, it was never there; it was artificially added by the \textit{regularizer}.

In this paper, we study several aspects related to the stability of bang-bang minimizers, and moreover to instability phenomena. These aspects  have many important uses, e.g., they help to deal with  uncertainty of data, to understand the technicalities appearing in the numerical methods, to regularize problems, etc.  In order to show that the phenomena studied here are independent  of  particular control systems, we work in very general framework that allows us to cover several optimal control problems. We illustrate this with a handful of examples.

Let us now comment a bit on the related literature.  For optimal control problems governed by ordinary differential equations, the stability analysis of bang-bang minimizers started with \cite{Ursulabb} for linear quadratic problems, and continued with \cite{Ursulabb2} for more general affine systems. After that, several refinements and analyses of numerical schemes came. In \cite{Seyd1, BjsCH}, the accuracy of implicit discretization schemes was analysed under growth assumptions on the switching function; these same assumptions  were used  in \cite{Khelifa,Prei} to prove the convergence of gradient methods. In  \cite{CQS,VQreg} the stability of the first order necessary conditions was studied by means of the\textit{ metric regularity property}. In \cite{Bimetric,Osmo} more \textit{natural} assumptions than previous papers were introduced to generate results about stability and Lipschitz rates of convergence; these assumptions involved $L^1$-growths, similar to the classic coercivity condition, but with some modifications. Recently, in \cite{Mpc,Sozopol}, the \textit{metric subregularity} of the optimality mapping was used to prove the accuracy of the \textit{model predictive control} algorithm.

The stability analysis of bang-bang minimizers for problems constrained by partial differential equations started with \cite{Soac,Hinzebb}, where elliptic optimal control problems were considered. Since then, there are several papers dealing other type of problems with bang-bang minimizers; see, e.g., \cite{CCJ,Elliptic,NiFE,DWQN}. For  parabolic problems, we mention  \cite{Par}, where a study on the accuracy of variational discretization  was carried out;  and  \cite{ParCasas}, where stability with respect to initial data was analysed. In \cite{Casasfluids},  a fully discrete scheme is proposed and analyzed for a velocity tracking problem with bang-bang controls. Finally, we comment on the recent paper \cite{Idriss}, where bilinear problems with  $L^1$--$L^\infty$ constraints are considered; the authors proved that under certain hypotheses, the optimal controls must be bang-bang (our setting is different, but with a few changes applies to the problem considered there). We present a similar result regarding the bang-bang nature of optimal controls, but for linearly perturbed problems, see \Cref{Genericity} for more details.

We now describe the the organization of the paper and the contributions of each section.

In \Cref{S2}, we introduce the optimization problem and give the definition of the bang-bang property under consideration. The first result is the following characterization.
\begin{align*}
	\textit{A control is bang-bang iff every  sequence converging weakly to it converges strongly.}
\end{align*}
This result (\cref{cor:equiv}) is based on previous results from  \cite{Casasbang} and \cite{Visin}. Since  the proof is a bit involved and requires some external tools from measure theory and set-valued analysis, it is given in the Appendix. After that, we use the characterization to prove that strict local minimizers satisfy a growth condition iff they are bang-bang, see \Cref{dospt} for more  details and  references.

In \Cref{Genericity}, we prove that an affine problem is \textit{well-posed} in Tykhonov's sense iff it possesses a bang-bang  strict global minimizer.   We also use a smooth variational principle to prove that this situation is generic, yielding a result of the following form.
 \begin{align*}
 	\textit{Almost every linearly perturbed problem has a bang-bang strict global minimizer.}
 \end{align*}
In \Cref{Sstab}, we prove that for strict local minimizers, hemicontinuity  notions of stability under linear perturbations  coincide, and  all of them are equivalent to the bang-bang property; therefore, many undesirable  instability phenomena can occur when the minimizer is not bang-bang. This shows that the study of problems with singular arcs is truly different in nature from the pure bang-bang case; this can already been seen in previous publications dealing with the stability of \textit{bang-singular-bang} minimizers, see \cite{Felg2,Felgenhauer}. In \Cref{Ssubreg}, we continue the study of stability, but now  for the first order necessary condition. This done by means of a modification of the \textit{metric subregularity property}. We prove that for isolated critical points (isolated solutions of the first order necessary condition), the stability of the first order necessary conditions is equivalent to the bang-bang property. As an application of this, we prove a stability result concerning $L^p$-norm regularizations.

\Cref{Examples} is devoted to illustrate the applicability of our results by means of examples. The first example is an affine optimal control problems constrained by ordinary differential equations, being this the model for the theory here developed. The other examples are concerned with optimization problems constrained by partial differential equations; this includes a classical elliptic problem, and a velocity tracking one.

Finally,  Appendix A  gives  alternative proofs to some  of the results in \cite{Visin}, which to the best of the author's knowledge was not done before. We mention that some of the results in \cite{Visin} contain flaws, although not substantial ones; see the beginning  of  Appendix  A for more details.

\section{The abstract optimal control problem}\label{S2}

We consider an abstract optimization problem for which the bang-bang property can be defined. The feasible set resembles the usual \textit{control sets} appearing in optimal control theory.

\subsection{The model}
Let  $(X,\mathcal A,\mu)$ be a measure space and consider the set
\begin{align}\label{conset}
	\mathcal U := \{u \in {L^1(X)^m} : u(x)\in  U\quad \text{a.e.~in } X\},
\end{align}
where $U$ is a subset of $\mathbb R^m$. We consider the abstract optimization problem
\begin{align}\label{cost}
	\min_{u\in\mathcal U}\mathcal J(u),
\end{align}
where $\mathcal J:\mathcal U\to\mathbb R$ is a given real-valued function.  The set in (\ref{conset}) is called the \textit{feasible set}  and the function in (\ref{cost}) is called the \textit{objective functional}.

We consider problem (\ref{conset})--(\ref{cost}) under the following standing assumption.

\begin{assumption}\label{standingass}
	We require that the following statements hold.
		\begin{enumerate}
			\item $(X,\mathcal A,\mu)$ is a finite and nonatomic measure space;
			
			\item $U$ is a convex compact subset of $ \mathbb R^m$ that contains more than one element;
			
			\item $\mathcal J:\mathcal U\to\mathbb R$ is weakly sequentially continuous.
		\end{enumerate}
\end{assumption}

The previous assumption is very reasonable for affine optimal control problems. The reason is that in those problems, the control appears linearly (hence the name), and  the weak sequential continuity of the objective functional can be expected. Moreover, under \cref{standingass}, the feasible set enjoys many good properties.
\begin{proposition}\label{wcfs}
	The  feasible set $\mathcal U$ is a nonempty, convex and weakly sequentially compact subset of ${L^1(X)^m}$.
\end{proposition}
\begin{proof}
	It follows as a particular case of \cite[Theorem 6.4.23]{Anhand} that $\mathcal U$ is a nonempty convex  weakly compact subset of ${L^1(X)^m}$. From the Eberlein–\v Smulian Theorem, we can conclude that $\mathcal U$ must be weakly sequentially compact.
\end{proof}
From the previous proposition, for every sequence in the feasible set, we are able to choose a weakly convergent subsequence; which is extremely useful in existence arguments.
\begin{definition}
	Let $u^*\in\mathcal U$. We define the   minimality radius of $u^*$ as
	\begin{align*}
		\bar r_{u^*}:=\sup\left\lbrace \delta\ge0:\,\,\mathcal J(u^*)\le\mathcal J(u)\,\,\,\,\text{for all $u\in\mathcal U$ with $|u-u^*|_{{L^1(X)^m}}\le\delta$}  \right\rbrace.
	\end{align*}
	We say that $u^*$ is a local minimizer of problem (\ref{conset})--(\ref{cost}) if $\bar r_{u^*}>0$. We say that $u^*$ is a global minimizer of problem (\ref{conset})--(\ref{cost}) if $\bar r_{u^*}=+\infty$.
\end{definition}

From the weak sequential compactness of the feasible set and the weak sequential continuity of the objective functional, the existence of minimizers follows trivially.

\begin{proposition}
Problem  (\ref{conset})--(\ref{cost}) has at least one global minimizer. 
\end{proposition}

	\subsection{Bang-bang property}	
	
We give now a precise definition of the \emph{bang-bang property} for elements of $\mathcal U$. We denote by $\ext U$ the set of extreme points of $U$.

\begin{definition}
	We say that $u\in\mathcal U$ is bang-bang if 
	\begin{align*}
		u(x)\in \ext U\quad\text{for a.e. $x\in X$.} 
	\end{align*}
\end{definition}

Elements of the feasible set with the bang-bang property are of general interest because they saturate the pointwise constraints. For example, in the particular case of a convex polytope, bang-bang elements take only values on the vertices of the polytope almost everywhere. Recall that a convex polytope is the convex hull of finitely many points. In the $2$-dimensional and $3$-dimensional  cases, convex polytopes  are exactly polygons and polyhedra, respectively.

We now will state a couple of results concerned with the sequences converging weakly to bang-bang elements. The following result was proved first in \cite[Corollary 2]{Visin} using \cite[Theorem 1]{Visin}; however there are some inconsistencies in some of the proofs of \cite{Visin}. In the Appendix, we comment on some of the flaws in  \cite{Visin}, and  give an alternative proof of this result and other related ones.

\begin{proposition}
	Let $u^*\in\mathcal U$ be bang-bang and $\{u_n\}_{n=1}^\infty\subset\mathcal U$ be a sequence. If  $u_n\rightharpoonup u^*$ weakly in ${L^1(X)^m}$, then $|u_n-u^*|_{{L^1(X)^m}}\to0$.
\end{proposition}
\begin{proof}
	This is a particular case of \cref{Visintinthm}.
\end{proof}

The phenomenon described in the previous proposition is well known in the calculus of variations, where weakly convergent minimizing sequences are usually also strongly convergent.  However, this is not always the case; the existence of bad behaved sequences follows from the following \textit{weak clustering principle}. For any element of the feasible set without the bang-bang property, it is  possible to find a sequence in the feasible set converging to it such that the sequence clusters on a sphere of arbitrarily small radius.

\begin{proposition}\label{csequences}
	Let $u^*\in\mathcal U$ . If $u^*$ is not bang-bang, 
	there exists $\delta_0>0$ such that for every $\delta\in(0,\delta_0]$ there exists a sequence $\{u_n\}_{n\in\mathbb N}\subset\mathcal U$ with the following properties. 
	\begin{enumerate}
		\item $|u_{n}-u^*|_{{L^1(X)^m}}=\delta$ for all $n\in\mathbb N$; 
		
		\item $u_n\rightharpoonup u^*$ weakly in ${L^1(X)^m}$.
	\end{enumerate}
\end{proposition}
\begin{proof}
	This is a particular case of \cref{Clustering}.
\end{proof}

The two previous results can be combined to obtain the following characterization of the bang-bang property.
\begin{theorem}\label{cor:equiv}
	Let $u^*\in\mathcal U$. The following statements are equivalent.
	\begin{enumerate}
		\item $u^*$ is bang-bang.
		\item $u_n\rightharpoonup u^*$ weakly in ${L^1(X)^m}$ implies $|u_n-u^*|_{{L^1(X)^m}}\to 0$ for any sequence $\{u_n\}_{n\in\mathbb N}\subset\mathcal U$.
	\end{enumerate}
\end{theorem}

The next result was proved in \cite[Theorem 2.1]{Casasbang} for local minimizers, in the particular case when the constraints in (\ref{conset}) are box-like (in the one-dimensional case $m = 1$) and under the additional assumptions of separability and completeness of the measure space $(X,\mathcal A,\mu)$.
\begin{corollary}\label{VpCasas}
	Let $u^*\in\mathcal U$. Suppose that $u^*$ is not bang-bang. 
	Then, there exists $\delta_0 > 0$ such that for any $\delta \in (0,\delta_0]$ and
	for any $\varepsilon > 0$,
	there exists $u \in \mathcal U$
	with
	\begin{equation}
		|{u - u^*}|_{{L^1(X)^m}} = \delta\ \text{ and }\ \mathcal J(u)\le\mathcal J(u^*) + \varepsilon.
	\end{equation}
\end{corollary}
\begin{proof}
	This follows from \cref{csequences} and the weak  sequential continuity of the objective functional.  
\end{proof}

\subsection{Growth of the objective functional at strict local minimizers}\label{dospt}

We begin recalling the definition of strict minimality, both local and global.
\begin{definition}
	Let $u^*\in\mathcal U$. We define the  strict minimality radius of $u^*$ as
	\begin{align*}
		\hat r_{u^*}:=\sup\left\lbrace \delta\ge0:\,\mathcal J(u^*)<\mathcal J(u)\,\,\,\, \text{for all $u\in\mathcal U\setminus\{u^*\}$ with $|u-u^*|_{{L^1(X)^m}}\le\delta$}  \right\rbrace.
	\end{align*}
	We say that $u^*$ is a strict local minimizer of problem (\ref{conset})--(\ref{cost}) if $\hat r_{u^*}>0$. We say that $u^*$ is a strict global minimizer of problem (\ref{conset})--(\ref{cost}) if $\hat r_{u^*}=+\infty$.
\end{definition}

It was proved in \cite[Section 2.1]{Casasbang} that
\emph{no growth} of the objective functional at minimizer can occur if the minimizer is not bang-bang. The proof was given for H\"older-type growths, and they point out that the argument also works for more general type of growths. We give here the argument for completeness.

\begin{proposition}
	Let $u^*\in\mathcal U$. Suppose that there exist   $\delta>0$ and a function $\omega:(0,\infty)\to(0,\infty)$ such that 
	\begin{align*}
		\mathcal J(u)\ge \mathcal J(u^*)+\omega\big(|u-u^*|_{{L^1(X)^m}}\big)
	\end{align*}
	for all  $u\in\mathcal U\setminus\{u^*\}$ with $|u-u^*|_{{L^1(X)^m}}\le\delta$. Then $u^*$ is  a bang-bang strict local minimizer of problem (\ref{conset})--(\ref{cost}). Moreover,  $\delta\le\hat r_{u^*}$.
\end{proposition}
\begin{proof}
If $u\in\mathcal U\setminus\{u^*\}$ satisfies $|u-u^*|_{{L^1(X)^m}}\le\delta$, then
	\begin{align*}
		\mathcal J(u)\ge \mathcal J(u^*)+\omega(|u-u^*|_{{L^1(X)^m}})>\mathcal J(u^*).
	\end{align*}
	Thus, $u^*$ must be a strict local minimizer of problem (\ref{conset})--(\ref{cost}) and $\delta\le\hat r_{u^*}$.
	Suppose that $u^*$ is not bang-bang, and let $\delta_0$ be the positive number in Proposition \ref{csequences}. Then there exist $\eta\in(0,\min\{\delta,\delta_0\})$  and a sequence $\{u_n\}_{n\in\mathbb N}\subset\mathcal U$, satisfying  $|u_{n}-u|_{{L^1(X)^m}}=\eta$ for all $n\in\mathbb N$, such that $u_n\rightharpoonup u$ weakly in ${L^1(X)^m}$. Then, 
	\begin{align*}
		\mathcal J(u_n)\ge \mathcal J(u^*)+\omega(\eta)\quad\text{for all $n\in\mathbb N$.}
	\end{align*}
	Since $\mathcal J$ is weakly sequentially continuous, we get $\omega(\eta)\le0$. A contradiction.
\end{proof}

The converse of the previous result is also true, the objective functional must satisfy a growth condition at bang-bang strict local minimizers.

\begin{proposition}\label{growthbang}
	Let $u^*\in\mathcal U$ be a strict local minimizer of problem (\ref{conset})--(\ref{cost}). Suppose that $u^*$ is bang-bang. Then there exist $\delta\in(0,\hat r_{u^*})$ and a non-decreasing function $\omega:(0,\infty)\to(0,\infty)$ such that 
	\begin{align*}
		\mathcal J(u)\ge \mathcal J(u^*)+\omega\big(|u-u^*|_{{L^1(X)^m}}\big)
	\end{align*}
	for all $u\in\mathcal U\setminus\{u^*\}$ with $|u-u^*|_{{L^1(X)^m}}\le\delta$.
\end{proposition}
\begin{proof}
	Let $M:=\sup_{u\in\mathcal U}|u-u^*|_{L^1(X)^m}$ and $\delta\in(0,\min\{M,\hat r_{u^*}\})$ be arbitrary. Let  $\omega_{\delta}:(0,\delta]\to(0,\infty)$ be given by
	\begin{align*}
		\omega_\delta(\eta):=\inf\left\lbrace \mathcal J(u)-\mathcal J(u^*): u\in\mathcal U\quad\text{and}\quad \eta\le|u-u^*|_{{L^1(X)^m}}\le\delta \right\rbrace.
	\end{align*}
	By construction, $\omega$ is nonnegative and  non-decreasing. Suppose that there exists $\eta\in(0,\delta]$ such that $\omega(\eta)=0$. By definition of infimum, there would exist a sequence $\{u_n\}_{n\in\mathbb N}\subset\mathcal U$ such that
	\begin{align}\label{genineqref}
		\eta\le|u_n-u^*|_{L^1(X)^m}\le\delta\quad\text{and}\quad0<\mathcal J(u_n)-\mathcal J(u^*)\le\frac{1}{n}\quad\text{for all $n\in\mathbb N$}.
	\end{align}
	We can extract a subsequence $\{u_{n_k}\}_{k\in\mathbb N}$ of $\{u_{n}\}_{n\in\mathbb N}$ converging weakly in ${L^1(X)^m}$ to some $\hat u\in\mathcal U$. Since $\mathcal J$ is weakly sequentially continuous, from (\ref{genineqref}), we get $\mathcal J(\hat u)=\mathcal J(u^*)$. Since $u^*$ is a strict local minimizer and
	\begin{align*}
		|\hat u-u^*|_{{L^1(X)^m}}\le\liminf_{k\to\infty}|u_{n_k}-u^*|_{{L^1(X)^m}}\le \delta<\hat r_{u^*},
	\end{align*}
	we conclude $\hat u=u^*$. This implies that $\{u_{n_k}\}_{k\in\mathbb N}$ converges weakly to $u^*$ in ${L^1(X)^m}$, and as $u^*$ is bang-bang, $\{u_{n_k}\}_{k\in\mathbb N}$ must converge to $u^*$ in ${L^1(X)^m}$; a contradiction. The result follows defining $\omega:(0,\infty)\to(0,\infty)$ as
	$\omega( \eta ) := \omega_\delta( \min( \eta, \delta ))$.
\end{proof}


\section{Genericity of the bang-bang property}\label{Genericity}

\subsection{The Radon--Nikod\'ym property}
One of the drawbacks of Bochner integration theory is that the Radon–Nikod\'ym Theorem fails to hold in general.  Sets for which this result still holds define an important class  that has been studied extensively; see, e.g., the specialized book \cite{GeoRNP}.

In this subsection, we give a short review of these sets in the particular case where the underlying Banach space is ${L^1(X)^m}$. We begin with the definition of dentability; for the general definition  see \cite[Definition 6.3.2]{VarTech} or \cite[Section 1.7.1]{Penot2013}.
\begin{definition}
	Let $\mathcal W$ be a nonempty subset of  ${L^1(X)^m}$. The elements of  $\{S(\mathcal W,\xi,\delta):\, \xi\in {L^\infty(X)^m},\,\delta>0\}$ are called slices of $\mathcal W$, where  
	\begin{align*}
		S\big(\mathcal W,\xi,\delta\big):=\{\,w\in\mathcal W: \,  \xi w\le \inf_{v\in\mathcal W}\xi v+\delta\}\quad \text{for $\xi\in {L^\infty(X)^m}$ and $\delta>0$}.
	\end{align*}
	We say that the subset $\mathcal W$ of $L^1(X)^m$ is dentable if it admits arbitrarily small slices, i.e., 
	 for every $\varepsilon>0$ there exist $\delta>0$ and $\xi\in {L^\infty(X)^m}$ such that $\diam S(\mathcal W,\xi,\delta)\le\varepsilon$.
\end{definition}

We now give the definition of the Radon--Nikod\'ym Property (RNP). There are many different ones, as the property has been characterized in many ways, see, e.g., the book \cite{GeoRNP} or the survey \cite{HuffGeoRNP} for geometrical characterizations. Here we give a definition based on dentability of sets, see \cite[Definition 6.3.3]{VarTech} or \cite[Section 1.7.1]{Penot2013}.

\begin{definition}
	A subset $\mathcal V$ of ${L^1(X)^m}$ has the Radon--Nikod\'ym Property  if every nonempty bounded subset $\mathcal W$ of $\mathcal V$ is dentable.
\end{definition}

We mention that the family of sets having the RNP is quite rich, it includes all reflexive spaces, see \cite[Corollary 4.1.5]{GeoRNP}, in particular the $L^p$-spaces for $p\in(1,\infty)$. There are non-reflexive spaces without the RNP, such as $L^1([0,1])$, see  \cite[Example 2.1.2]{GeoRNP}. However, subsets of $L^1$-spaces might possess the property even if the whole space does not have it. For example, it is known that weakly compact convex subsets of Banach spaces have the RNP, see \cite[Theorem 3.6.1]{GeoRNP}.
\begin{proposition}\label{RNPconset}
	The feasible set $\mathcal U$ has the Radon--Nikod\'ym Property.
\end{proposition}
\begin{proof}
			By \cref{wcfs}, $\mathcal U$ is weakly sequentially compact, and hence by the Eberlein–\v Smulian Theorem, weakly compact. We can use then \cite[Theorem 3.6.1]{GeoRNP} to conclude that $\mathcal U$ has the Radon--Nikod\'ym Property.
\end{proof}

\subsection{Strong minimizers and well-posedness}
This subsection is devoted to recall one of the classical concepts of well-posedness, the \textit{Tikhonov} one. In order to do so,  we first recall the definition of \textit{strong minimality}, see \cite[Definition 6.3.4]{VarTech}.
\begin{definition}
	Let $\mathcal F:\mathcal U\to\mathbb R$ be a functional and  $u^*\in\mathcal U$ a minimizer of $\mathcal F$. We say that $u^*$ is a strong minimizer if 
	\begin{align*}
		\mathcal F(u_n)\to\mathcal F(u^*)\,\,\,\,\text{implies}\,\,\,\, |u_n-u^*|_{L^1(X)^m}\to 0
	\end{align*}
for any sequence $\{u_{n}\}_{n\in\mathbb N}\subset\mathcal U$.
\end{definition}

Using \cref{cor:equiv}, we can easily characterize strong minimizers in terms of the bang-bang property. 

\begin{proposition}\label{strongmin}
	Let $\mathcal F:\mathcal U\to\mathbb R$ be a weakly sequentially continuous functional and  $u^*\in\mathcal U$. The following statements are equivalent. 
	\begin{enumerate}
		\item $u^*$ is a strong minimizer of $\mathcal F$.
		\item $u^*$ is a bang-bang strict minimizer of $\mathcal F$.
	\end{enumerate}
\end{proposition}
\begin{proof}
	Suppose that $u^*$ is a strong minimizer of $\mathcal F$, and let $\{u_{n}\}_{n\in\mathbb N}\subset\mathcal U$ be a sequence converging weakly to $u^*$ in ${L^1(X)^m}$. Since $\mathcal F$ is weakly sequentially continuous, $\mathcal F(u_n)\to\mathcal F(u^*)$, and thus $u_n\to u^*$. We can use \cref{cor:equiv} to conclude that $u^*$ must be bang-bang. It is clear that strong minimizers are strict.
	
	Conversely, suppose that $u^*$ is a bang-bang strict minimizer, and  let $\{u_{n}\}_{n\in\mathbb N}\subset \mathcal U$ be a sequence such that $\mathcal F(u_n)\to\mathcal F(u^*)$. Let $\{u_{n_k}\}_{k\in\mathbb N}$ be a subsequence of $\{u_{n}\}_{n\in\mathbb N}$.  We can extract a subsequence $\{u_{n_{k_j}}\}_{j\in\mathbb N}$ of $\{u_{n_k}\}_{k\in\mathbb N}$  converging weakly in ${L^1(X)^m}$ to some $\hat u\in\mathcal U$. Since $\mathcal F(u_{n_{k_j}})\to\mathcal F(u^*)$, we obtain $\mathcal F(\hat u)=\mathcal F(u^*)$, but since $u^*$ is a strict minimizer, we must have $\hat u=u^*$. Since every subsequence of $\{u_n\}_{n\in\mathbb N}$ has further a subsequence that converges weakly to $u^*$, we conclude that $u_n \rightharpoonup u^*$ weakly in ${L^1(X)^m}$; but then, by  \cref{cor:equiv}, $u_{n}\to u^*$ strongly in ${L^1(X)^m}$.
\end{proof}

We now give the definition of well-posedness, see \cite[Section 1.7]{Penot2013}.
\begin{definition}
	Let $\mathcal F:\mathcal U\to\mathbb R$. The optimization problem 
	\begin{align*}
		\min_{u\in\mathcal U}\mathcal F(u)
	\end{align*}
is said to be well-posed if $\mathcal F$ has strong minimizer.
\end{definition}

We can use \cref{strongmin} to see how the well-posedness of problem (\ref{conset})--(\ref{cost}) is related to the bang-bang property.

\begin{corollary}\label{wpbb}
	Problem (\ref{conset})--(\ref{cost}) is well-posed if and only if it possesses a bang-bang strict global minimizer.
\end{corollary}

\subsection{Stegall's principle and the bang-bang property}
In nonlinear analysis and
optimization,  the notions of variational principle, perturbation and well-posedness are intrinsically related. The concept of genericity is of  fundamental importance in the interplay of these notions. In topology, a generic property is usually one that holds on a dense open set, however this can be a very strong requirement; for example, the irrational numbers are somehow generic among the real numbers, but they do not conform a open dense subset of the real line. 
 More generally, a generic property is one that holds on a residual set, being the dual concept a meager set. We recall now the definition of residuality in the particular case of the space ${L^\infty(X)^m}$.
\begin{definition}
	Let $\Theta$ be a subset of ${L^\infty(X)^m}$. We say that $\Theta$ is residual if there exists a countable family of open dense sets $\{D_n\}_{n\in\mathcal N}\subset {L^\infty(X)^m}$ such that 
	\begin{align*}
		\Theta=\bigcap_{n\in\mathbb N} D_n.
	\end{align*}
A set is said to be meager if it is the complement of a residual set.
\end{definition}
From Baire Category Theorem, it is clear that residual subsets of ${L^\infty(X)^m}$ are dense. We now give the definition of genericity, see \cite[Section 7]{Penot2013}. 
\begin{definition}
	A subset of ${L^\infty(X)^m}$ is said to be generic if it contains a residual set. A property is said to be generic if it holds on a generic set.
\end{definition}
We come now to a classic in variational analysis,  Stegall's  principle. This is a type of smooth variational principle over sets with the RNP.  It appeared first in \cite[pp. 174-176]{Stegall}, see \cite[Theorem 1.153]{Penot2013} or \cite[Theorem 6.3.5]{VarTech} for book references. We state now a version of the theorem, as a particular case of  \cite[Theorem 5]{Lassonde}.
\begin{theorem}\label{Stegallvp}
	The set 
	\begin{align*}
		\{\xi\in {L^\infty(X)^m}:\,\, \text{problem\, $\min_{u\in\mathcal U}\{\mathcal J(u)-\xi u\}$ is well-posed}\}
	\end{align*}
	is generic. In other words, the well-posedness of linearly perturbed versions  of problem (\ref{conset})--(\ref{cost}) is generic.
\end{theorem}
\begin{proof}
	By \cref{RNPconset}, the feasible set $\mathcal U$ has the RNP. Since the objective functional $\mathcal J$ is weakly sequentially continuous, it is in particular sequentially  continuous, and hence continuous. We can then employ  the implication $(v)\implies (iii)$ of \cite[Theorem 5]{Lassonde} to conclude the result.
\end{proof}

We are going now to reformulate the previous theorem in a way that makes transparent that the bang-bang property is generic. 

\begin{theorem}\label{bbg}
	The set
	\begin{align*}
	\left\lbrace \xi\in {L^\infty(X)^m}:\,\,\,	\mathcal J-\xi \,\,\,\text{has a bang-bang strict global minimizer}\right\rbrace 
	\end{align*}
	 is generic. In other words, the existence of bang-bang strict global minimizers of  linearly perturbed problems of problem (\ref{conset})--(\ref{cost}) is generic.
\end{theorem}
\begin{proof}
	Let $\Lambda$ be the set of $\xi\in {L^\infty(X)^m}$ such that $\mathcal J-\xi$ has a bang-bang strict global minimizer.
	By  \cref{Stegallvp}, there exists a residual set $\Theta\subset {L^\infty(X)^m}$ such that
	problem $\min_{u\in\mathcal U}\{\mathcal J(u)-\xi u\}$ is well-posed for $\xi\in\Theta$. 
	By \cref{strongmin}, $\Theta$ is contained in $\Lambda$; thus $\Lambda$ is generic.
\end{proof}
\begin{corollary}\label{bbSteg}
	For every $\varepsilon>0$ there exists $\xi\in L^\infty(X)^m$  with $ |\xi|_{{L^\infty(X)^m}}\le\varepsilon$ such that $\mathcal J-\xi$ has a bang-bang strict global minimizer.
\end{corollary}
\begin{proof}
	By \cref{bbg}, there exists a residual set $\Theta\subset {L^\infty(X)^m}$ such that for any $\xi\in\Theta$, $\mathcal J-\xi$ has a bang-bang strict global minimizer. The result follows from  Baire Category Theorem, as it implies that  residual sets of ${L^\infty(X)^m}$ are dense.
\end{proof}
\section{Stability under linear perturbations}\label{Sstab}
 We study the relation of linear perturbations to problem (\ref{conset})--(\ref{cost}) and the bang-bang property. We begin by describing the solution mappings that we consider.  We  give new characterizations of the bang-bang property in terms of hemicontinuity of these mappings. Finally, in the last subsection, we reformulate the stability properties  to make clearer their meaning.

\subsection{Solution mappings}
When studying stability of an optimization problem, it is often possible to prove that  perturbed problems have solutions in a neighborhood of a reference solution. This type of solutions can serve as a first approach for the stability analysis.

Given $\xi\in {L^\infty(X)^m}$, $u^*\in\mathcal U$ and $\gamma>0$, we  consider the optimization problem
\begin{align*}
	\mathcal P_{u^*,\gamma}(\xi):\quad\min_{}\left\lbrace \mathcal J(u)-\xi u:\,\,\,\,u\in\mathcal U\,\,\,\,\text{with}\,\,\,\,|u-u^*|_{{L^1(X)^m}}\le\gamma\right\rbrace.
\end{align*}
Before advancing further, let us mention that each problem has at least one solution.

\begin{proposition}\label{exisper}
	Each problem $\mathcal P_{u^*,\gamma}(\xi)$ has at least one minimizer, i.e, there exists $u=u_{u^*,\gamma,\xi}$ such that 
	\begin{align*}
		\mathcal J(u)-\xi u\le \mathcal J(w)-\xi w\quad\text{for all $w\in\mathcal U$ with $|w-u^*|_{L^1(X)^m}\le \gamma$}.
	\end{align*}
\end{proposition}
\begin{proof}
	The set $\mathcal V:=\{u\in\mathcal U:\,\, |u-u^*|_{L^1(X)^m}\le\gamma\}$ is closed and convex, hence weakly closed. As by \cref{wcfs}, $\mathcal U$ is weakly sequentially compact, so is $\mathcal V$. Clearly, each map $\mathcal J-\xi:\mathcal V\to \mathbb R$ is weakly sequentially continuous; therefore, problem $\mathcal P_{u^*,\gamma}(\xi)$ must have at least one global minimizer.
\end{proof}

With each problem $\mathcal P_{u^*,\gamma}(\xi)$, we associate a \textit{localized solution mapping}. This is a  set-valued mapping, denoted by $\mathcal S_{u^*,\gamma}:{L^\infty(X)^m}\twoheadrightarrow L^1(X)^m$ and  given by
\begin{align*}
	\mathcal S_{u^*,\gamma}(\xi):=\left\lbrace u\in\mathcal U: \text{$u$ is a minimizer of problem $\mathcal P_{u^*,\gamma}(\xi)$} \right\rbrace.
\end{align*}

\begin{proposition}\label{nonemptysm}
	Each set-valued mapping $\mathcal S_{u^*,\gamma}$ takes nonempty closed values.
\end{proposition}
\begin{proof}
	It follows from \cref{exisper} that each $S_{u^*,\gamma}$ takes nonempty values. It follows from the weak sequential continuity of the objective functional that each $\mathcal S_{u^*,\gamma}$ takes closed values.
\end{proof}
We can recover the usual solution mappings from the localized ones.
The \textit{local solution mapping}  $\mathcal S_{\text{loc}}: {L^\infty(X)^m}\to {L^1(X)^m}$ is given by
\begin{align*}
	\mathcal S_{\text{loc}}(\xi):=\left\lbrace u\in\mathcal U:\, u\in \mathcal S_{u,\gamma}(\xi) \,\,\,\,\text{for some $\gamma>0$}\right\rbrace.
\end{align*}
Each set $\mathcal S_{\text{loc}}(\xi)$  consists of the local minimizers of $\mathcal J-\xi$ on $\mathcal U$.
In the same fashion, we define the \textit{global solution mapping}  $\mathcal S_{\text{gbl}}: {L^\infty(X)^m}\to {L^1(X)^m}$ by 
\begin{align*}
	\mathcal S_{\text{gbl}}(\xi):=\left\lbrace u\in\mathcal U:\, u \in \mathcal S_{u,\gamma}(\xi) \,\,\,\,\text{for all $\gamma>0$}\right\rbrace.
\end{align*}
Similarly, each set $\mathcal S_{\text{gbl}}(\xi)$  consists of the global minimizers of $\mathcal J-\xi$ on $\mathcal U$.

\subsection{Hemicontinuity}
It is now time to study the continuity properties of the solution mappings described in the previous subsection. We do this by means of the notion of hemicontinuity; we use standard definitions, see \cite[Definition 17.2]{Infdim}.  The term semicontinuity is used sometimes instead of hemicontinuity; see e.g., \cite[Section 1.4.1]{Frankowska}.

We begin studying the lower hemicontinuity properties of the mappings in relation with the bang-bang property; in order to do so, we employ the following sequential characterization of lower hemicontinuity, see \cite[Theorem 17.21]{Infdim}.
\begin{proposition}
	Let $\mathcal S:L^\infty(X)^m\twoheadrightarrow L^1(X)^m$ be a set-valued mapping. The following statements are equivalent.
	\begin{enumerate}
		\item $\mathcal S$ is lower hemicontinuous at $0$.
		
		\item For every sequence  $\{\xi_n\}_{n\in\mathbb N}\subset L^\infty(X)^m$  converging to $0$ in $L^\infty(X)^m$ and every $u\in\mathcal S(0)$, there exists a  subsequence $\{\xi_{n_k}\}_{k\in\mathbb N}$ of $\{\xi_n\}_{n\in\mathbb N}$ and a sequence $\{u_k\}_{k\in\mathbb N}\subset L^1(X)^m$ such that
		\begin{align*}
		u_k\in \mathcal S(\xi_{n_k})\,\,\,\,\text{for all $k\in\mathbb N$}\quad\text{and}\quad	u_k\to u\,\,\,\text{in $L^1(X)^m$}.
		\end{align*}
	\end{enumerate}
\end{proposition}

We are now ready for our first result. The proof consist of two main ingredients, the weak clustering principle (\cref{csequences}) and the construction of adequate perturbations. For the latter, we use the celebrated Hahn-Banach Theorem.
\begin{proposition}\label{imbang}
	Let $u^*\in\mathcal U$ be a strict local minimizer of problem (\ref{conset})--(\ref{cost}). Suppose that there exists $\gamma>0$ such that $\mathcal S_{u^*,\gamma}$ is lower hemicontinuous at $0$. Then $u^*$ is bang-bang.
\end{proposition}
\begin{proof}
	Suppose that $u^*$ is not bang-bang. By \cref{csequences}, there exists a positive number $\delta<\min\{\hat r_{u^*},\gamma\}$ and sequence $\{u_{n}\}_{n\in\mathbb N}\subset\mathcal U$ such that 
	\begin{align*}
		u_{n}\rightharpoonup u^*\,\,\,\,\text{weakly in $L^1(X)^m$}\quad\text{and}\quad\text{$|u_n-u^*|_{L^1(X)^m}=\delta$ for all $n\in\mathbb N$}.
	\end{align*}
	By the Hahn-Banach Theorem, for each $n\in\mathbb N$ there exists  $\xi_n\in L^\infty(X)^m$ such that
	\begin{align*}
		\xi_{n}(u_n-u^*)=|\xi_n|_{L^\infty(X)^m}|u_n-u^*|_{L^1(X)^m}\quad\text{and}\quad |\xi_{n}|_{L^\infty(X)^m}=\frac{4}{\delta}|\mathcal J(u_n)-\mathcal J(u^*)|
	\end{align*}
	for all $n\in\mathbb N$. Since $\delta<\hat r_{u^*}$, it follows that  $\mathcal J(u^*)<\mathcal J(u_n)$ for all $n\in\mathbb N$, and hence $|\xi_{n}|_{L^\infty(X)^m}>0$ for all $n\in\mathbb N$. Also, from the weak sequential continuity of the objective functional, it follows that $\xi_{n}\to 0$ in $L^\infty(X)^m$.

		Now, as $\mathcal S_{u^*,\gamma}$ is lower hemicontinuous at $0$ and $\mathcal S_{u^*,\gamma}(0)=\{u^*\}$, there exists a subsequence $\{\xi_{n_k}\}_{k\in\mathbb N}$ of $\{\xi_{n}\}_{n\in\mathbb N}$ and a  sequence $\{w_{k}\}_{k\in\mathbb N}$ converging to $u^*$ such that  $w_k\in\mathcal S_{u^*,\gamma}(\xi_{n_k})$ for all $k\in\mathbb N$. Then
	\begin{align}\label{ine1}
		\mathcal J(w_{k})-\xi_{n_k} w_k\le \mathcal J(u_{n_k})-\xi_{n_k} u_{n_k}\quad \text{for all $k\in\mathbb N$.}
	\end{align}
	Since $w_k\to u^*$ in $L^1(X)^m$, there exists $k_0\in\mathbb N$ such that $|w_k-u^*|_{L^1(X)^m}\le 2^{-1}\delta<\hat r_{u^*}$ for $k\ge k_0$, and hence $\mathcal J(u^*)\le \mathcal J(w_k)$ for $k\ge k_0$. Combing this with \cref{ine1}, we get
	\begin{align*}
		-|\mathcal J(u^*)-\mathcal J(u_{n_k})|=\mathcal J(u^*)-\mathcal J(u_{n_k})\le\mathcal J(w_{k})-\mathcal J(u_{n_k})\le \xi_{n_k}\big(w_k-u_{n_k}\big)
	\end{align*}
for all $k\ge k_0$.  Now, by construction of the sequence $\{\xi_n\}_{n\in\mathbb N}$, 
\begin{align*}
	\frac{\delta}{4}|\xi_{n_k}|_{L^\infty(X)^m}
	&=|\mathcal J(u^*)-\mathcal J(u_{n_k})|
	\ge
	\xi_{n_k}\big(u_{n_k}-w_{k}\big)\\
	&=\xi_{n_k}\big(u_{n_k}-u^*\big)+\xi_{n_k}\big(u^*-w_k\big)
	\\
	&\ge |\xi_{n_k}|_{L^\infty(X)^m}\big(\delta-|w_k-u^*|_{L^1(X)^m}\big)\ge \frac{\delta}{2}|\xi_{n_k}|_{L^\infty(X)^m}
\end{align*}
	for all $k\ge k_0$. This yields a contradiction.
\end{proof}

We proceed now to analyze the  upper hemicontinuity of localized solution mappings. The following sequential characterization will be of use, see \cite[Theorem 17.20]{Infdim}.

\begin{proposition}\label{uhc}
	Let $\mathcal S:L^\infty(X)^m\twoheadrightarrow L^1(X)^m$ be a set-valued mapping. The following statements are equivalent.
	\begin{enumerate}
		\item $\mathcal S$ is upper hemicontinuous at $0$ and $\mathcal S(0)$ is compact.
		
		\item If $\{\xi_{n}\}_{n\in\mathbb N}\subset L^\infty(X)^m$ and $\{u_n\}_{n\in\mathbb N}\subset L^1(X)^m$ are sequences such that
		\begin{align*}
			u_n\in \mathcal S(\xi_{n})\,\,\,\,\text{for all $n\in\mathbb N$}\quad\text{and}\quad	\xi_n\to 0\,\,\,\text{in $L^\infty(X)^m$},
		\end{align*}
		then the sequence  $\{u_n\}_{n\in\mathbb N}$ has a limit point in $\mathcal S(0)$.
	\end{enumerate}
\end{proposition}

It turns out that the bang-bang property can imply upper hemicontinuity at a point when the  localized solution mappings are single valued at that point.

\begin{proposition}
	Let $u^*\in\mathcal U$ be a strict local minimizer of problem (\ref{conset})--(\ref{cost}) and $\gamma\in(0,\hat r_{u^*})$. If $u^*$ is bang-bang, then $\mathcal S_{u^*,\gamma}$ is upper hemicontinuous at $0$.
\end{proposition}
\begin{proof}
 Let $\{\xi_{n}\}_{n\in\mathbb N}\subset L^\infty(X)^m$ be a sequence such that $\xi_n\to 0$ in $L^\infty(X)^m$ and $\{u_n\}_{n\in\mathbb N}\subset \mathcal U$ a sequence such that $u_n\in\mathcal S_{u^*,\gamma}(\xi_n)$ for all $n\in\mathbb N$.
  Then 
\begin{align}\label{ine2}
	\mathcal J(u_n)-\xi_n u_n\le \mathcal J(u^*)-\xi_n u^*\quad\text{for all $n\in\mathbb N$.}
\end{align}
We can extract a subsequence $\{u_{n_k}\}_{k\in\mathbb N}$ of   $\{u_n\}_{n\in\mathbb N}$ converging weakly to some $\hat u\in\mathcal U$ in $L^1(X)^m$. 
 Taking limit in \cref{ine2}, we conclude that $\mathcal J(\hat u)\le \mathcal J(u^*)$.  As $\gamma<\hat r_{u^*}$ and 
 \begin{align*}
 	|\hat u-u^*|_{L^1(X)^m}\le\liminf_{k\to 0}|u_{n_k}-u^*|_{L^1(X)^m}\le \gamma,
 \end{align*}
  we conclude that $\hat u=u^*$. Hence,  $u_{n_k}\rightharpoonup u^*$ weakly in $L^1(X)^m$, and as $u^*$ is bang-bang,  $u_{n_k}\to u^*$ in $L^1(X)^m$; thus $\{u_n\}_{n\in\mathbb N}$ has a limit point in $\mathcal S_{u^*,\gamma}(0)=\{u^*\}$.
\end{proof}

We can now put together the two previous results in a single theorem characterizing the bang-bang property in terms of hemicontinuity.

\begin{theorem}\label{thmhemi}
	Let $u^*\in\mathcal U$ be a strict local minimizer of problem (\ref{conset})--(\ref{cost}) and $\gamma\in(0,\hat r_{u^*})$. The following statements are equivalent.
	\begin{enumerate}
		\item $\mathcal S_{u^*,\gamma}$ is upper hemicontinuous at $0$.
		
		\item $\mathcal S_{u^*,\gamma}$ is lower hemicontinuous at $0$.
		
		\item $u^*$ is bang-bang.
	\end{enumerate}
\end{theorem}
\begin{proof}
	The implication $(ii)\implies(iii)$ follows from \cref{imbang} and the implication  $(iii)\implies(i)$ follows from \cref{uhc}. We proceed then to prove the implication $(i)\implies(ii)$. We first observe that $\mathcal S_{u^*,\gamma}(0):=\{u^*\}$ since $\gamma<\hat r_{u^*}$. Let $\{\xi_n\}_{n\in\mathbb N}\subset L^\infty(X)^m$ be any sequence converging to zero in $L^\infty(X)^m$.  There exists a sequence $\{u_{n}\}_{n\in\mathbb N}\subset\mathcal U$ satisfying $u_n\in \mathcal S_{u^*,\gamma}(\xi_n)$ for all $n\in\mathbb N$. 
	As $\mathcal S_{u^*,\gamma}$ is upper hemicontinuous at $0$, $\{u_n\}_{n\in\mathbb N}$ has $u^*$ as limit point; hence there exists a subsequence $\{u_{n_k}\}_{k\in\mathbb N}$ of $\{u_n\}_{n\in\mathbb N}$ converging to $u^*$ in $L^1(X)^m$. We conclude that $\mathcal S_{u^*,\gamma}$ is lower hemicontinuous at $0$.

\end{proof}

We can also give a characterization in terms of the global solution mapping.

\begin{corollary}\label{corgbl}
	Let $u^*\in\mathcal U$ be a strict global minimizer of problem (\ref{conset})--(\ref{cost}). The following statements are equivalent.
	\begin{enumerate}
		\item $\mathcal S_{\text{gbl}}$ is upper hemicontinuous at $0$.
		
		\item $\mathcal S_{\text{gbl}}$ is lower hemicontinuous at $0$.
		
		\item $u^*$ is bang-bang.
	\end{enumerate}
\end{corollary}
\begin{proof}
	Choose $\gamma>0$ such that $\mathcal U\subset\{u\in L^1(X)^m:\,\, |u-u^*|_{L^1(X)^m}\le\gamma\}$. Then 
	\begin{align*}
		\mathcal S_{\text{gbl}}(\xi)=\mathcal S_{u^*,\gamma}(\xi)\quad\text{for all $\xi\in L^\infty(X)^m$}.
	\end{align*}
	Consequently, the result follows from \cref{thmhemi}.
\end{proof}

We close the subsection with the following result relating the bang-bang property to the upper hemicontinuity at zero of the local solution mapping.

\begin{corollary}
	Suppose that problem (\ref{conset})--(\ref{cost}) has unique local minimizer $u^*\in\mathcal U$. If $\mathcal S_{\text{loc}}$ is upper hemicontinuous at $0$, then $u^*$ is bang-bang.
\end{corollary}
\begin{proof}
	Observe that $\mathcal S_{\text{loc}}(0)=\mathcal S_{\text{gbl}}(0)=\{u^*\}$.
	Let $\{\xi_n\}_{n\in\mathbb N }\subset L^\infty(X)^m$ be sequence converging to zero in $L^\infty(X)^m$. There exists a sequence $\{u_{n}\}_{n\in\mathbb N}$  such that $u_n\in\mathcal S_{\text{gbl}}(\xi_n)$ for all $n\in\mathbb N$.  In particular, $u_n\in\mathcal S_{\text{loc}}(\xi_n)$ for all $n\in\mathbb N$. As $\mathcal S_{\text{loc}}$ is assumed to be upper hemicontinuous at $0$, it follows that $\{u_n\}_{n\in\mathbb N}$ has a subsequence $\{u_{n_{k}}\}_{k\in\mathbb N}$ such that $u_{n_k}\to u^*$ in $L^1(X)^m$. Thus, we conclude that $\mathcal S_{\text{gbl}}$ is lower hemicontinuous. Then, by \cref{corgbl}, $u^*$ must be  bang-bang.
\end{proof}

\subsection{Local stability of linear perturbations}

In the previous subsection, the stability properties of problem (\ref{conset})--(\ref{cost}) were studied in terms of hemicontinuity, whose sequential characterization involves subsequences and limit points.
In this subsection, we restate these results in  terms that make clearer their meaning.

\begin{definition}\label{locdef}
Let $u^*\in\mathcal U$ be a local minimizer of problem (\ref{conset})--(\ref{cost}). We say that problem (\ref{conset})--(\ref{cost}) is locally stable at $u^*$ if there exists $\gamma>0$ with the property that for every $\varepsilon>0$ there exists $\delta>0$ such that 
\begin{align}\label{Staprope}
	|\xi|_{{L^\infty(X)^m}}<\delta \qquad\text{implies}\qquad |u-u^*|_{{L^1(X)^m}}<\varepsilon
\end{align}
for any $u\in\mathcal S_{u^*,\gamma}(\xi)$ and any $\xi\in L^\infty(X)^m$.
We define the stability radius of $u^*$ as the positive number $\hat\gamma_{u^*}:=\sup\left\lbrace\gamma>0:\,\text{(\ref{Staprope}) holds}\right\rbrace$.
\end{definition}

The definition of local stability says that small perturbations in $L^\infty(X)^m$ should imply that all  solutions of localized perturbed  problems  be close to $u^*$ in $L^1(X)^m$. This agrees with the common understanding of stability.

We come now to a characterization of local stability in terms of the bang-bang property; the result follows  easily from the  hemicontinuity properties studied in the  previous subsection.

\begin{theorem}\label{lclystable}
	Let $u^*\in\mathcal U$ be a  local minimizer of problem (\ref{conset})--(\ref{cost}). The following statements are equivalent.
	\begin{enumerate}
		\item Problem (\ref{conset})--(\ref{cost}) is  locally stable at $u^*$.
		
		\item $u^*$ is a bang-bang strict local minimizer of problem (\ref{conset})--(\ref{cost}).
	\end{enumerate}
	Moreover, if  problem (\ref{conset})--(\ref{cost}) is  locally stable at $u^*$, then  $\hat\gamma_{u^*}=\hat r_{u^*}$. 
\end{theorem}
\begin{proof}
	Suppose that problem (\ref{conset})--(\ref{cost}) is locally stable at $u^*$. It follows immediately from the  definition of local stability that $u^*$ is strict local minimizer and that $\hat\gamma_{u^*}\le \hat r_{u^*}$.  It is also easy to see that $\mathcal S_{u^*,\gamma}$ is upper hemicontinuous at $0$ for any $\gamma\in(0,\hat\gamma_{u^*})$, and hence that $u^*$ is bang-bang. 
	
	Suppose now that $u^*$ is a bang-bang strict local minimizer of problem (\ref{conset})--(\ref{cost}).  Suppose that $\hat \gamma_{u^*}<\hat r_{u^*}$, and let $\gamma\in(\hat\gamma_{u^*},\hat r_{u^*})$. Then there exist $\varepsilon>0$, a sequence  $\{\xi_{n}\}_{n\in\mathbb N}\subset L^\infty(X)^m$ converging to zero in $L^\infty(X)^m$ and a sequence  $\{u_{n}\}_{n\in\mathbb N}\subset\mathcal U$ such that 
	\begin{align*}
		 |u_n-u^*|_{L^1(X)^m} \ge\varepsilon\quad\text{and}\quad u_n\in\mathcal S_{u^*,\gamma}(\xi_n)\quad\text{for all $n\in\mathbb N$.}
	\end{align*}
	 By \cref{uhc}, $\mathcal S_{u^*,\gamma}$ is upper hemicontinuous at $0$. Consequently, $\{u_{n}\}_{n\in\mathbb N}$ must have a subsequence $\{u_{n_k}\}_{k\in\mathbb N}$ such that $u_{n_k}\to u^*$ in $L^1(X)^m$; this yields a contradiction. We conclude that problem (\ref{conset})--(\ref{cost}) is locally stable at $u^*$ and that $\hat\gamma_{u^*}=\hat r_{u^*}$. 
\end{proof}

We now pass to the global analysis. For the sake of clarity and transparency, we give a definition that  reflects the intuitive understanding of global stability.

\begin{definition}\label{locdefgbl}
	Let $u^*\in\mathcal S_{\text{gbl}}(0)$. We say that problem (\ref{conset})--(\ref{cost}) is globally stable at $u^*$ if for every $\varepsilon>0$ there exists $\delta>0$ such that 
	\begin{align*}
		|\xi|_{{L^\infty(X)^m}}<\delta \quad\text{implies}\quad |u-u^*|_{{L^1(X)^m}}<\varepsilon
	\end{align*}
	for any $\xi\in L^\infty(X)^m$ and  $u\in\mathcal S_{\text{gbl}}(\xi)$.
\end{definition}

From \cref{lclystable}, we can deduce immediately the following result.

\begin{corollary}
	Let $u^*\in\mathcal U$ be a global minimizer of problem (\ref{conset})--(\ref{cost}). The following statements are equivalent.
	\begin{enumerate}
		\item Problem (\ref{conset})--(\ref{cost}) is  globally stable at $u^*$.
		
		\item $u^*$ is a bang-bang strict global minimizer of problem (\ref{conset})--(\ref{cost}).
	\end{enumerate}
\end{corollary}
\begin{proof}
	Let $\gamma_0>0$ such that $\mathcal U\subset\{u\in L^1(X)^m:\,\, |u-u^*|_{L^1(X)^m}\le\gamma_0\}$. Then 
	\begin{align*}
		\mathcal S_{\text{gbl}}(\xi)=\mathcal S_{u^*,\gamma}(\xi)\quad\text{for all $\xi\in L^\infty(X)^m$ and all $\gamma\ge \gamma_0$}.
	\end{align*}
	The result follows then from  \cref{lclystable}.
\end{proof}
\section{Stability of the first-order necessary condition}\label{Ssubreg}

We study now the stability with respect to  perturbations of  the first-order necessary condition of problem (\ref{conset})--(\ref{cost}).  In general, the first-order necessary condition (the local Pontryagin principle for optimal control problems) can be written as an inclusion, the so-called optimality system. Thus the stability properties of the first-order necessary condition can be directly analyzed from this inclusion. We will employ a concept of stability based on  the so-called \textit{strong metric subregularity} property. The definition of this property was given first in \cite[Definition 5.1]{RSsubReg}. The recent paper \cite{Smsr} gives a good overview of the utility of this property and its role in  variational analysis and optimization. For book references, see \cite[Lecture 12]{LecVar} or  \cite[Section 3I]{Imsm}.
\subsection{The first-order necessary condition}
We recall briefly the fist order necessary conditions for problem (\ref{conset})--(\ref{cost}).
We will employ the classic notion of (first-order) Gateaux differentiability.
\begin{definition}\label{Gatdiff}
The objective functional ${\mathcal J}:\mathcal U\to\mathbb R$ is said to be  Gateaux differentiable if for every $u\in\mathcal U$ there exists $d\mathcal J(u)\in L^\infty(X)^m$ such that
	\begin{align*}
	d{\mathcal J}(u)v=\lim_{\varepsilon\to0^+}\frac{{\mathcal J}(u+\varepsilon v)-{\mathcal J}(u)}{\varepsilon}\quad\text{for all $v\in L^1(X)^m$ with $u+v\in\mathcal U$.}
	\end{align*}
\end{definition}
	The first-order necessary condition is well known, see, e.g.,  \cite[pp. 11-13]{LecVar}.
\begin{lemma}\label{Fonc}
	Suppose that $\mathcal J:\mathcal U\to\mathbb R$ is Gateaux differentiable. If $u^*\in\mathcal U$ is a local minimizer of problem (\ref{conset})--(\ref{cost}), then
	\begin{align}\label{Necconever0}
		d\mathcal J(u^*)(u-u^*)\ge0\quad\text{for all $u\in\mathcal U$.}
	\end{align}
\end{lemma}

In order to talk about the stability of the first-order necessary conditions, during this section, we will of course assume that the objective functional is Gateaux differentiable; and moreover, a  \textit{weak-strong} continuity property on the derivative. The following assumption is supposed to hold throughout the remainder of this section.
\begin{assumption}\label{A2}
	The following statements hold.
	\begin{enumerate}
		\item The objective functional  $\mathcal J:\mathcal U\to\mathbb R$ is Gateaux differentiable;
		
		\item the mapping $\mathcal Q:\mathcal U\to L^\infty(X)^m$ given by $\mathcal Q(u):=d\mathcal J(u)$ is weakly-strongly sequentially continuous,
			i.e.,
			\begin{align*}
				\text{$u_n \rightharpoonup u$ weakly in $L^1(X)^m$\,\,\,\, implies\,\,\,\,	$\mathcal Q(u_n) \to \mathcal Q(u)$ in $L^\infty(X)^m$}
			\end{align*}
			for any sequence $\{u_n\}_{n \in \mathbb N} \subset \mathcal U$ and any $u\in\mathcal U$.
	\end{enumerate}
\end{assumption}

In analogy with optimal control,  we write $\sigma_{u}:=\mathcal Q(u)$ for each  $u\in\mathcal U$. The mapping $\mathcal Q:\mathcal U\to L^\infty(X)^m$ in $(ii)$ of \cref{A2} is called the \textit{switching mapping}.

Let us recall that the normal cone to $\mathcal U$ at $u^* \in \mathcal U$ is given by 
\begin{align*}
	N_{\mathcal U}(u^*):=\left\lbrace \xi\in L^\infty(X)^m:\langle \xi,u-u^*\rangle\le0\quad\text{for all $u\in  \mathcal U$}\right\rbrace.
\end{align*}
 For $u^* \in L^1(X)^m \setminus \mathcal U$, we set $N_{\mathcal U}(u^*) = \emptyset$.
We can then rewrite the first-order necessary condition as the inclusion
\begin{align}\label{Necconever}
	0\in \sigma_{u}+N_{\mathcal U}(u).
\end{align}

The correspondence $\Phi:\mathcal U\twoheadrightarrow L^\infty(X)^m$  given by $\Phi(u):=\sigma_u+N_{\mathcal U}(u)$ is called \textit{the optimality mapping}.
We now give a definition concerned with inclusion (\ref{Necconever}).

\begin{definition}
	Let $u^* \in \mathcal U$ be given.
	\begin{enumerate}
		\item $u^*$ is said to be a critical point of problem (\ref{conset})--(\ref{cost}) if $0\in\Phi(u^*)$;
		
		\item  $u^*$ is said to be  a locally isolated critical point of problem (\ref{conset})--(\ref{cost}) if there exists $\delta>0$ such that 
		\begin{align}\label{iscrimin}
			\text{$0\in\Phi(u)$ implies $u=u^*$ for all $u\in\mathcal U$ with $|u-u^*|_{L^1(X)^m}\le\delta$}.
		\end{align}
	\end{enumerate}
The critical radius of $u^*$ is given by $	\check r_{u^*}:=\sup\left\lbrace \delta>0: \,\text{\cref{iscrimin} holds}\right\rbrace$.
\end{definition}

\subsection{Subregularity of the optimality mapping}
We are now going to study the stability of inclusion (\ref{Necconever}) under perturbations. We will employ the following definition based on the notion of \textit{strong metric subregularity}, see  \cite[Introduction]{Smsr}.
\begin{definition}
	Let $u^*$ be a critical point of problem (\ref{conset})--(\ref{cost}). We say that the optimality mapping $\Phi:\mathcal U\to L^\infty(X)^m$ is strongly subregular at $u^*$ if there exists $\kappa>0$ with the property that for every $\varepsilon>0$ there exists $\delta>0$ such that
	\begin{align}\label{Stapropefo}
	|\xi|_{L^\infty(X)^m }<\delta \quad\text{implies}\quad |u-u^*|_{{L^1(X)^m}}<\varepsilon
	\end{align}
	for any $u\in\mathcal U$ with $|u-u^*|_{{L^1(X)^m}}\le\kappa$ and any $\xi\in\Phi(u)$.
	We define the  subregularity radius of $u^*$ as $\hat\kappa_{u^*}:=\sup\left\lbrace\kappa>0:\text{ property (\ref{Stapropefo}) holds}\right\rbrace$.
\end{definition}

We state now a trivial consequence of the definition of subregularity. 
\begin{proposition}\label{genprop0}
	Let $u^*\in\mathcal U$. Then $\hat\kappa_{u^*}\le\check r_{u^*}$. In particular, if the optimality mapping is strongly subregular at $u^*$, then $u^*$ is a locally isolated critical point of problem (\ref{conset})--(\ref{cost}).
\end{proposition}

The following theorem states the necessity of the bang-bang property.

\begin{lemma}\label{lemma0}
	 Let $u^*\in\mathcal U$ be a local minimizer of problem (\ref{conset})--(\ref{cost}). If the optimality mapping is strongly subregular at $u^*$, then $u^*$ is bang-bang.
\end{lemma}
\begin{proof}
The subregularity of the optimality mapping at $u^*$ clearly implies that problem (\ref{conset})--(\ref{cost}) is locally stable at $u^*$. By \cref{lclystable}, $u^*$ must be bang-bang.
\end{proof}

We arrive now to the main result of this section. 
\begin{theorem}\label{subregthm}
	Let $u^*\in\mathcal U$ be a local minimizer of problem (\ref{conset})--(\ref{cost}). The following statements are equivalent.
	\begin{enumerate}
		\item The optimality mapping is strongly subregular at $u^*$.
		
		\item  $u^*$ is a bang-bang locally isolated critical point of problem (\ref{conset})--(\ref{cost}).
	\end{enumerate}
Moreover, if the optimality mapping is strongly subregular at $u^*$, then $\hat\kappa_{u^*}=\check r_{u^*}$.
\end{theorem}
\begin{proof}
	The implication $(i)\implies(ii)$ follows from \cref{genprop0} and \cref{lemma0}.

	Let us prove now the implication  $(ii)\implies(i)$. From  \cref{genprop0}, we have $\hat \kappa_{u^*}\le \check r_{u^*}$.
	Towards a contradiction, suppose that  $\hat \kappa_{u^*}<\check r_{u^*}$ and let $\delta\in(\hat\kappa_{u^*},\check r_{u^*})$. Then there exist a  number $\varepsilon>0$, a sequence $\{\xi_{n}\}_{n\in\mathbb N}\subset L^\infty(X)^m $ converging to zero in $L^\infty(X)^m$ and a sequence  $\{u_{n}\}_{n\in\mathbb N}\subset\mathcal U$ such that 
	\begin{align*}
	\delta\ge|u_n-u^*|_{{L^1(X)^m}}\ge\varepsilon\quad\text{and}\quad \xi_n\in \sigma_{u_n}+N_{\mathcal U}(u_n)\quad\text{for all $n\in\mathbb N$.}
	\end{align*}
	We can extract a subsequence $\{u_{n_k}\}_{k\in\mathbb N}$ of  $\{u_n\}_{n\in\mathbb N}$ converging weakly to some $\hat u\in\mathcal U$. Now, since each $u_{n_k}$ satisfies $\xi_{n_k}\in\sigma_{u_{n_k}}+
	N_{\mathcal U}(u_{n_k})$, taking limit, we obtain $0\in\sigma_{\hat u}+N_{\mathcal U}(\hat u)$. As $u^*$ is a locally isolated critical point of problem (\ref{conset})-(\ref{cost}) and
	\begin{align*}
		|\hat u-u^*|_{{L^1(X)^m}}\le\liminf_{n\to\infty}|u_{n_k}-u^*|_{{L^1(X)^m}}\le \delta<\hat r_{u^*},
	\end{align*}
	we conclude that $\hat u=u^*$, and hence that $u_{n_k}\rightharpoonup u^*$ weakly in ${L^1(X)^m}$. But, as $u^*$, is bang-bang, it must be that $u_{n_k}\to u^*$ in ${L^1(X)^m}$; a contradiction. Then $\hat\kappa_{u^*}=\check r_{u^*}$.
\end{proof}

\subsection{An application: $p$-regularization}	
We give now an application of the subregularity property concerned with \textit{Tykhonov  regularizations}.

Let $p>1$ be given. For each $\eta>0$, we consider the following \textit{regularized optimization problem}.
\begin{align*}
	\mathcal P_p(\eta):\,\,\,\,\min_{u\in\mathcal U}\left\lbrace \mathcal J(u)+\frac{\eta}{p}\int_{X}|u(x)|^p\,d\mu(x).\right\rbrace.
\end{align*}

From subregularity, we can conclude the following regularization result.
\begin{theorem}\label{regloc}
	Let $u^*\in\mathcal U$ be a critical point of problem (\ref{conset})-(\ref{cost}). Suppose that optimality mapping is strongly subregular at $u^*$. Then for every $\varepsilon>0$ there exists $\eta_\varepsilon>0$ such that  
	\begin{align*}
		\eta<\eta_\varepsilon \,\,\,\,\, \text{implies}	\,\,\,\,\,|u_{\eta}-u^*|_{L^1(X)^m}<\varepsilon
	\end{align*}
	for any local minimizer $u_{\eta}\in\mathcal U$ of $\mathcal P_p(\eta)$ satisfying $|u_{\eta}-u^*|_{L^1(X)^m}<\hat \kappa$.
\end{theorem}
\begin{proof}
	Let $\varepsilon>0$ be arbitrary, and 
let $u_\eta\in\mathcal U$ be any local minimizer of $\mathcal P_p(\zeta)$ satisfying $|u_{\eta}-u^*|_{L^1(X)^m}<\hat \kappa$. The first order necessary condition can be written as
\begin{align*}
	d\mathcal J(u_\eta)(u-u_\eta)+\eta\int_X |u_\eta(x)|^{p-2} u_\eta(x)\cdot (u(x)-u_\eta(x))\, d\mu(x)\ge0\quad\text{ $\forall u\in\mathcal U$.}
\end{align*}
This can be rewritten as the inclusion $\xi_\eta\in \sigma_{u_\eta}+N_{\mathcal U}(u_\eta)$, 
where $\xi_\eta\in L^\infty(X)^m$ is given by $\xi_\eta(x):=\eta|u_\eta(x)|^{p-2} u_\eta(x)$.
By strong subregularity of the optimality mapping at $u^*$, there exists $\delta_\varepsilon>0$ such that if $|\xi_\eta|_{L^\infty(X)}<\delta_\varepsilon$, then $|u_\eta-u^*|_{L^1(X)}<\varepsilon$. It is enough then to take $\eta_\varepsilon:=\delta_\varepsilon\big[\sup U\big]^{1-p}$
\end{proof}

We omit the case $p=1$ as it is a bit more involved, however and identical result can be accomplished by means of subregularity.

\section{Examples of the theory}\label{Examples}
We provide three examples of the class of  optimal control problems studied in this paper. The first one is constrained by an ordinary differential equation, and the second and third ones by partial differential equations.

\subsection{Affine optimal control problems constrained by ordinary differential equations}
As a canonical example of the theory developed in previous sections, we consider the  \textit{affine optimal control  problem}  given by
\begin{align}\label{costexa1}
	\min_{u\in\mathcal U}\left\lbrace s_T(y_{u}(T))+\int_{0}^T\Big[g_0\big(t,y_{u}(t)\big)+\sum_{i=1}^m g_i\big(t,y_{u}(u)\big) u_i(t)\Big]\, dt \right\rbrace,
\end{align}
where  for each control $u=(u_1,\dots,u_m)\in\mathcal U$, there is a unique state $y_{u}:[0,T]\to\mathbb R^n$ satisfying
\begin{align}\label{consetexa1}
		\dot y_{u}=f_0(\cdot,y_{u})+\sum_{i=1}^m f_i(\cdot,y_{u}) u_i,
		\qquad
		y_u(0)=y_0 .
\end{align}
We give below the specifications of problem (\ref{costexa1})--(\ref{consetexa1})  and the  technical details.

The number $T>0$ is the (fixed) \textit{time horizon}.  The underlying measure space $([0,T],\mathcal A_{[0,T]}, \mathcal L)$ consists of the $\sigma$-algebra $\mathcal A_{[0,T]}$  of Lebesgue measurable subsets of $[0,T]$, and $\mathcal L:\mathcal A\to\mathbb R$ the Lebesgue measure on $[0,T]$. 

For a compact convex set $U\subset\mathbb R^m$, the feasible set takes the form
\begin{align*}
	\mathcal U=\left\lbrace u \in L^1(0,T)^m :\,\, u(t)\in U\,\,\, \text{for a.e. $t\in[0,T]$} \right\rbrace.
\end{align*}
The functions $f_0,\dots,f_m:[0,T]\times\mathbb R^n\to \mathbb R^n$ are Carath\' eodory and satisfy
\begin{align*}
	 \esssup_{t\in[0,T]}\sup_{x\in\mathbb R^n}f_i(t,x)<\infty\quad\text{and}\quad\esssup_{t\in[0,T]}\sup_{x_1,x_2\in\mathbb R^n}\frac{|f_{i}(t,x_1)-f_{i}(t,x_2)|}{|x_1-x_2|}<\infty
\end{align*}
for each $i\in\{0,\dots,m\}$.
 These functions and the \textit{initial datum} $y_0\in\mathbb R^n$ determine the dynamic for each control in the following way. For each $u\in\mathcal U$, by the classical global existence theorem (see, e.g., \cite[Theorem 2.1.1]{Bressan}),  there exists a unique state $y_u\in W^{1,1}\big([0,T]\big)^n$ satisfying (\ref{consetexa1}), i.e., 
 \begin{align*}
 	y_{u}(t)=y_0+\int_{0}^t\Big[f_0(s,y_{u}(s))+\sum_{i=1}^mf_i(s,y_{u}(s)) u_i(s)\Big]\, ds\quad \forall t\in[0,T].
 \end{align*}
The  \textit{cost functions} $g_0,\dots g_m:[0,T]\times\mathbb R^n\to\mathbb R$ are Carath\' eodory and the \textit{scrap function} $s_T:\mathbb R^n\to \mathbb R$ is continuous.
The objective functional $\mathcal J:\mathcal U\to\mathbb R$ is given by
\begin{align*}
		\mathcal J(u)= s_T(y_{u}(T))+\int_{0}^T\Big[g_0\big(t,y_{u}(t)\big)+\sum_{i=1}^m g_i\big(t,y_{u}(u)\big) u_i(t)\Big]\, dt.
\end{align*}
Clearly, problem (\ref{costexa1})--(\ref{consetexa1})  trivially satisfies $(i)$ and $(ii)$ of Assumption \ref{standingass}. Item $(iii)$ is well-known to hold for these type of problems. Indeed, it follows easily from the integral form of the  Gr\"onwall inequality that the  \textit{input-output mapping}  $\mathcal S:\mathcal U\to C\big([0,T]\big)^n$, given by $\mathcal S(u)=y_u$, is weakly-strongly sequentially continuous. From this and the affine structure of the problem, a few calculations yield that the objective functional $\mathcal J:\mathcal U\to\mathbb R$ is weakly sequentially continuous.

\subsection{Elliptic optimal control problems}

We consider now an \textit{elliptic optimal control problem}.
\begin{align}\label{costexa2}
	\min_{u\in\mathcal U}\left\lbrace\int_{\Omega} L(x,y_{u}(x))\, dx \right\rbrace,
\end{align}
where for each control $u\in\mathcal U$, there is a unique state $y_{u}:\Omega\to\mathbb R$ satisfying
\begin{align}\label{consetexa2}
		-\Delta y_u+d(\cdot,y_u)=u \quad \text{in $\Omega$},
		\quad\,
		y_u=0 \quad \text{on $\Omega$}.
\end{align}
The specifications of problem (\ref{costexa2})--(\ref{consetexa2})  and the  data assumptions are given below.

The underlying measure space $(\Omega,\mathcal A_{\Omega}, \mathcal L)$ consists of  a bounded Lipschitz domain $\Omega$, the $\sigma$-algebra $\mathcal A_{\Omega}$  of Lebesgue measurable subsets of $\Omega$, and $\mathcal L:\mathcal A\to\mathbb R$ the Lebesgue measure on $\Omega$. 

For numbers $u_a,u_b\in\mathbb R$ with $u_a<u_{b}$,  the feasible set takes the form
\begin{align*}
	\mathcal U=\left\lbrace u \in L^1(\Omega):\,\, u(x)\in [u_a,u_b]\,\,\, \text{for a.e. $x\in\Omega$} \right\rbrace.
\end{align*}
The function $d:\Omega\times\mathbb R\to\mathbb R$ is Carath\' eodory, monotone nondecreasing with respect to the second variable and satisfies
\begin{align*}  \esssup_{x\in\Omega}\sup_{y\in K}|d(x,y)|<\infty
\end{align*}
for every compact set $K\subset \mathbb R$.
 For each control $u\in\mathcal U$, there exists a unique $y_u\in H_{0}^1(\Omega)\cap C(\bar\Omega)$ (see, e.g., 
 \cite[Theorem 2.7]{Casaselliptic}) satisfying (\ref{consetexa2}), i.e., 
\begin{align*}
	\int_{\Omega} \Big[\nabla y_{u}(x)\cdot\nabla \varphi(x)+d(x,y(x))\varphi(x)\Big]\,dx=\int_{\Omega}u(x) \varphi(x)\, dx\quad\forall \varphi\in H_0^1(\Omega).
\end{align*}
 The function $L:\Omega\times\mathbb R\to\mathbb R$ is Carath\'eodory and bounded by below.
The objective functional $\mathcal J:\mathcal U\to\mathbb R$ is then given by
\begin{align*}
	\mathcal J(u)=\int_{\Omega}L(x,y_{u}(x))\, dx.
\end{align*}
Clearly, problem (\ref{costexa2})--(\ref{consetexa2})  trivially satisfies $(i)$ and $(ii)$ of Assumption \ref{standingass}. One can easily prove that   the \textit{control-to-state mapping}  $\mathcal S:\mathcal U\to H_0^1(\Omega)\cap C(\Omega)$, given by $\mathcal S(u)=y_u$, is weakly-strongly sequentially continuous. For example,  using the arguments given in \cite[Lemma 2.5]{Elliptic} and  \cite[Proposition 2.12]{Elliptic}. From this and the form of the cost function, it is almost trivial that the objective functional $\mathcal J:\mathcal U\to\mathbb R$ is weakly sequentially continuous.

\subsection{A velocity tracking problem}
In this subsection, we discuss an optimization problem constrained by the Navier-Stokes equations. We skip technical details and rather point to the relevant references.

Let $\Omega\subset\mathbb R^2$ be a domain with boundary of class $C^2$ and $T>0$. Denote $Q:=\Omega\times(0,T)$, $\Sigma:=\partial\Omega\times(0,T)$ and let  $\bm u_a,\bm u_b:Q\to\mathbb R^2$ be bounded functions. The feasible set is
\begin{align}
	\mathcal U:=\left\lbrace \bm u:\Omega\to\mathbb R^2:\,\, \bm u_a(x,t)\le\bm u(x,t)\le \bm u_b(x,t)\,\,\,\,\text{for a.e. $(x,t)\in Q$} \right\rbrace.
\end{align}
For a given $\bm{y}_d\in L^\infty(Q)$, the optimal control is given by
\begin{align}
	\min_{u\in\mathcal U}\left\lbrace \frac{1}{2}\int_{0}^T\int_\Omega |\bm{y_u} - \bm y_{d}|^2 \, dx\, dt\right\rbrace.
\end{align}
subject to
\begin{align}\label{dynsys}
	\left\{ \begin{array}{l}
		\partial_t\bm{y_u}- \nu\Delta \bm{y_u}+ (\bm{y_u}\cdot\nabla)\bm{y_u} + \nabla \bm{p_u} = \bm u\,\,\,\,\text{in Q},\\
		\\ \dive \bm{y_u}=0\,\,\,\,\text{in Q},\,\, \bm {y_u}=0\,\,\,\,\text{on $\Sigma$},\,\,\, \bm{y_u}(\cdot,0)=\bm y_0\,\,\,\text{in $\Omega$}.
	\end{array}
	\right.
\end{align}
We refer the reader to  \cite[Section 2]{Casasfluids} for technical details of this problem. It was proved in \cite[Corollary 2.4]{Casasfluids} that the control-to-state operator maps weakly convergent sequences to strongly convergent ones. This can easily be used to conclude that weak sequential continuity of the objective functional.

\appendix
\section{Convergence under extremal conditions} 
There are spaces where weak convergence and strong convergence are equivalent, such spaces are said to have the Schur property; the canonical example being  the space $l^1(\mathbb N)$ of summable sequences. 

Unfortunately, general $L^1$-spaces  may not  possess the Schur property, being  the space ${L^1([0,1])}$ the classic counterexample. 
In \cite{Visin} several results regarding these type of properties for  particular sequences in general $L^1$-spaces were given, being probably the most important the following one.

\begin{theorem} \cite[Theorem 1]{Visin}
	Let $(X,\mathcal A,\mu)$ be a complete $\sigma$-finite measure space. Let $\{f_n\}_{n\in\mathbb N}\subset L^1(X)^m$ be a sequence of functions converging weakly in $L^1(X)^m$ to a function $f\in L^1(X)^m$. If 
	\begin{align*}
		\text{$f(x)$ is an extremal point of $K(x):=\clconv\Big(\{f(x)\}\cup\{f_n(x)\}_{n\in\mathbb N}\Big)$}
	\end{align*}
		 for a.e. $x\in X$, then $|f_n-f|_{L^1(X)^m}\to 0$.
\end{theorem}

 Nevertheless, the notation in the proof \cite[Theorem 1]{Visin} is sloppy and quite confusing sometimes. Moreover, it ignores many measurability issues.  In this paper, we need the following result, proved in  \cite{Visin} as a corollary of  \cite[Theorem 1]{Visin}.
 \begin{corollary}\cite[Corollary 2]{Visin}\label{corynq}
 		Let $(X,\mathcal A,\mu)$ be a  complete $\sigma$-finite measure space, $F:X\twoheadrightarrow\mathbb R^m$ a set-valued mapping taking nonempty compact convex values, and $f\in {L^1(X)^m}$ such that $f(x)\in\ext F(x)$ for a.e. $x\in X$. Let  $\{f_n\}_{n\in\mathbb N}\subset {L^1(X)^m}$ be a sequence of functions such that $f_n(x)\in F(x)$ for a.e. $x\in X$. If $f_n\rightharpoonup f$  weakly in ${L^1(X)^m}$, then $f_n\to f$ in ${L^1(X)^m}$.
\end{corollary}
 
 We use the first subsection of this appendix to provide an alternative proof of \cite[Corollary 2]{Visin}. We mention that the result we give (\cref{Visintinthm}) is somewhat different in the assumptions. We assume that the set-valued mapping in \cref{corynq} takes compact convex values, which suffices for the purposes of this paper. So in  that regard the result given here is weaker. On the positive side, though not a significant improvement, our result drops some measure-theoretical assumptions made in  \cite{Visin}  such as completeness and sigma-finiteness of the underlying measure, see \cref{Visintinthm} below for the details.  To the best of the author's knowledge, there is not an alternative proof of \cite[Theorem~1]{Visin}  or \cite[Corollary 2]{Visin} in the literature.
 
 We also point out that the proof of \cite[Proposition 1]{Visin} contains a flaw, namely it is assumed that the set of extreme points of a convex set is closed;  which is not true for subsets of $\mathbb R^m$ with $m\ge3$.
 In the second subsection of this appendix we give a different version of a result related to \cite[Proposition 1]{Visin} in general measure spaces (the result given there is for $X$ being a measurable subset of the Euclidean space). Finally, we comment that our \cref{Clustering} has as a corollary the result \cite[Proposition 1]{Visin} under some additional assumptions.

\subsection{The extremal Schur's property}
We first prove some preparatory lemmas. The first one is concerned with the \textit{equi-integrability} of weakly precompact sets; we refer the reader to  \cite[Section 15]{Voigt2020} for the definition of equi-integrability and its consequences.

\begin{lemma}\label{lem:equi-integrable}
	Let $(X,\mathcal A,\mu)$ be a measure space and $\mathcal V $ be a weakly relatively compact subset of ${L^1(X)^m}$.
	The following statements hold.
	\begin{enumerate}
		\item For every $\varepsilon > 0$ there exists $\delta > 0$ 
		such that 
		\begin{align*}
			\mu(A)<\delta\quad\text{implies}\quad \sup_{f\in\mathcal V}\int_{A}|f(x)|d\mu(x)<\varepsilon\quad\text{for all $A\in\mathcal A$.}
		\end{align*}
		\item For every $\varepsilon > 0$ there exists a finite measure set $B \in \mathcal A$ 
		such that
		\begin{align*}
			\sup_{f \in \mathcal V} \int_{X \setminus B} |f| d\mu< \varepsilon.
		\end{align*}
	\end{enumerate}
\end{lemma}
\begin{proof}
	By the Dunford-Pettis Theorem, the set $\mathcal V$ is a family of equi-integrable functions; see \cite[Theorem 15.4]{Voigt2020}. We can then apply \cite[Lemma 15.3]{Voigt2020} to conclude the result.
\end{proof}

 A proof of the next lemma, for the particular case of finite measure spaces, can be found in \cite[Proposition 2.3.38]{NonlinearAnalysis}. We give a proof for the general case.
\begin{lemma}\label{lem:weak_gives_strong_nonneg}
	Let $(X,\mathcal A,\mu)$ be a measure space,  $\{f_n\}_{n=1}^\infty \subset {L^1(X)}$ a sequence and  $f\in {L^1(X)}$. Suppose that 
	\begin{align*}
	f_n\rightharpoonup f\text{ weakly in ${L^1(X)}$}\quad\text{and}\quad	f(x)\le\liminf_{n\to\infty}f_{n}(x)\quad\text{for a.e. $x\in X$.}
	\end{align*}
	Then, $|f_n-f|_{{L^1(X)^m}}\to0$.
\end{lemma}
\begin{proof}
	For each $n\in\mathbb N$, define $v_n:=f_n-f$.  Then $\liminf_{n\to\infty} v_n\ge0$ a.e. in $X$, and since $v_n\rightharpoonup 0$ weakly in ${L^1(X)}$, the set $\{v_{n}: n\in\mathbb N\}$ is a weakly relatively compact subset of ${L^1(X)}$. 
	
	Let $\varepsilon > 0$ be arbitrary. Due to \cref{lem:equi-integrable},
	there exists a finite measure set $B \in\mathcal A$  and $\delta>0$ such that
	\begin{align*}
		\sup_{n\in\mathbb N} \int_{X \setminus B} |v_n(x)| d\mu(x)< \frac{\varepsilon}{4}\quad\text{and}\quad \sup_{\substack{A\in\mathcal A \\ \mu(A)<\delta}}\sup_{n\in\mathbb N} \int_{A} |v_n(x)| d\mu(x) < \frac{\varepsilon}{4}
	\end{align*}
	By Egorov's theorem, we can find $A \subset B$ with $\mu(A)<\delta$ such that the convergence
	\begin{equation*}
		\lim_{n \to \infty} \inf_{k \ge n} v_k
		=
		\liminf_{n \to \infty} v_n
	\end{equation*}
	is uniform in $B \setminus A$. 	In particular, there exists $n_1\in \mathbb N$
	such that
	$v_n\ge -\varepsilon (8\mu(B \setminus A))^{-1}$
	a.e.\ on $B \setminus A$ for all $n\ge n_1$.
	It follows that
	${|v_n|} \le v_n + \varepsilon(4\mu(B \setminus A))^{-1}$ a.e. in $B \setminus A$ for all $n\ge n_1$.
	By definition of weak convergence,
	there exists $n_2\in\mathbb N$ such that
	$\int_{B \setminus A} v_n d \mu < \varepsilon/4$
	for all $n \ge n_2$.
	Putting all together, 
	\begin{align*}
		\int_X {|v_n(x)|} d\mu(x)
		&=
		\int_{X \setminus B} {|v_n(x)|} d\mu(x)
		+
		\int_{B \setminus A} {|v_n(x)|} d\mu(x)+
		\int_A {|v_n(x)|} d\mu(x)\\
		&<\frac{\varepsilon}{4}
		+
		\int_{B \setminus A}\Big[v_n(x)+\frac{\varepsilon}{4\mu(B\setminus A)}\Big]\,d\mu(x)	+\frac{\varepsilon}{4}\le\varepsilon
	\end{align*}
	 for all $n \ge\max\{n_1,n_2\}$. This shows that  $|v_n|_{{L^1(X)^m}} \to 0$.
\end{proof}

Recall that given a convex set $C\in\mathbb R^m$ and $u\in C$, the set 
	\begin{align*}
		N_{C}(u):=\{\nu\in\mathbb R^m:\, \nu\cdot (v-u)\le0\quad\text{for all $v\in C$}\}
	\end{align*}
	is the normal cone to $C$ at $u$. We now address some measurability issues arising from the normal cone. These are of  technical nature and will be needed later on. 
	We use the standard definitions of measurability of set-valued mappings, see  \cite[Definition 6.2.1]{Anhand}.
\begin{lemma}\label{measissues}
		Let $(X,\mathcal A)$ be a measurable space,  $F:X\twoheadrightarrow \mathbb R^m$ a measurable set-valued mapping taking nonempty compact convex values, and $f:X\to\mathbb R^m$ a measurable function such that $f(x)\in F(x)$ for a.e. $x\in X$.
		The set-valued mapping 
	\begin{align*}
		x\mapsmto N_{F(x)}\big(f(x)\big)\quad \text{from $X$ to $\mathbb R^m$ is measurable. }
	\end{align*}
\end{lemma}
\begin{proof}
	By \cite[Corollary 18.15]{Infdim}, there exists a countable family $\{f_{n}\}_{i\in\mathbb N}$ of measurable functions $f_n:X\to\mathbb R^m$ such that 
	\begin{align*}
		F(x)=\overline{\{f_{n}(x):n\in\mathbb N\}}\quad\text{for all $x\in X$.}
	\end{align*}
	Let $K:X\twoheadrightarrow\mathbb R^m$ be given by
	\begin{align*}
			K(x):=\left\lbrace \xi\in\mathbb R^m: \xi\cdot\big(f_{n}(x)-f(x)\big)\le0\quad\text{for all $n\in\mathbb N$}\right\rbrace.
	\end{align*}
The measurability of $K$ follows from \cite[Theorem 3]{Measconv}. Finally, observe that 
\begin{align*}
	\text{$K(x)=N_{F(x)}\big(f(x)\big)$ for all $x\in X$}. 
\end{align*}
\end{proof}

We now give a geometrical construction that will allow to extend \cref{lem:weak_gives_strong_nonneg} to the higher dimensional case. We proceed inductively, taking care of the measurability issues arising in the process.

\begin{lemma}\label{norconlem}
	Let $(X,\mathcal A)$ be a measurable space,  $F:X\twoheadrightarrow \mathbb R^m$ a measurable set-valued mapping taking nonempty compact convex values such that $0\in\ext{F(x)}$ for a.e. $x\in X$.  Then there exist measurable functions $e_1,\dots, e_m: X\to\mathbb R^m$ such that for a.e. $x\in X$,  $\{e_1(x)\dots,e_{m}(x)\}$ is an orthonormal basis of  $\mathbb R^m$ and
	\begin{align}\label{l1app}
	0\in e_{i}(x)+N_{F(x)\cap H_{i}(x)}(0)\quad\text{for all $i\in\{1,\dots, m\}$},
	\end{align}
		where $H_1(x):=\mathbb R^m$ and 
		\begin{align*}
			H_{i+1}(x):=\{v\in\mathbb R^m:\,e_{j}(x)\cdot v=0\quad\text{for all $j\in\{1,\dots,i\}$}\}\quad\text{for $i=1,\dots,m-1$}.
		\end{align*}
\end{lemma}
\begin{proof}
	We argue by induction.
	The case $i=1$ follows trivially from \cref{measissues}, and the Kuratowski–Ryll-Nardzewski Selection Theorem, see \cite[Theorem 18.13]{Infdim}.
	Now, let $i\in\{1,\dots,m-1\}$ and suppose that there exists a measurable function $e_{i}:X\to\mathbb R^m$ such that
	\begin{align*}
		-e_i(x)\in N_{F(x)\cap H_i(x)}(0)\cap H_{i}(x)\quad\text{and}\quad |e_i(x)|=1\quad\text{for a.e. $x\in X$}.
	\end{align*}
	By \cite[Corollary 3.6]{Measconv}, the set valued mapping $x \mapsmto H_{i+1}(x)$ is measurable, hence as $F$ is a compact-valued measurable mapping, it follows that  the mapping $x\mapsmto F(x)\cap H_{i+1}(x)$ is measurable, see \cite[Proposition 6.2.21]{Anhand}. As for a.e. $x\in X$,  $f(x)$ is an extreme point of $F(x)$, it follows that $0$ is a boundary point of $F(x)\cap H_{i+1}(x)$. By \cref{measissues}, the mapping $x\mapsmto N_{F(x)\cap H_{i+1}(x)}(0)$ is measurable. Define $K_{i+1}:X\twoheadrightarrow\mathbb R^m$ by 
	\begin{align*}
		K_{i+1}(x):=N_{F(x)\cap H_{i+1}(x)}(0)\cap H_{i+1}(x)\cap \{\nu\in\mathbb R^m:|\nu|=1\}
	\end{align*}
	By \cite[Proposition 6.2.21]{Anhand}, $K_{i+1}$ is measurable; moreover it is clear that it takes nonempty closed values. We can then use the  Kuratowski–Ryll-Nardzewski Selection Theorem to conclude the existence of a measurable selection  $e_{i+1}$ of $K_{i+1}$. This completes the induction step.
	
	Finally, we note that as $e_{i}(x)\in H_{i}(x)\cap\{\nu\in\mathbb R^m:|\nu|=1\}$ for a.e. $x\in X$ and all $i\in\{1,\dots, m\}$, the set $\{e_1(x),\dots,e_{m}(x)\}$ generates an orthonormal basis for a.e. $x\in X$.
\end{proof}

We now proceed to a merely technical lemma that allows to localize an argument given in the main theorem.

\begin{lemma}\label{preplemapp}
	Let $(X,\mathcal A,\mu )$ be a measure space,  $F:X\twoheadrightarrow \mathbb R^m$ a measurable set-valued mapping taking nonempty compact convex values such that $0\in\ext{F(x)}$ for a.e. $x\in X$. Let $\{f_{n}\}_{n\in\mathbb N}\subset {L^1(X)^m}$ be a sequence converging weakly to zero in ${L^1(X)^m}$.  Consider functions $e_1,\dots,e_m\in {L^\infty(X)^m}$ satisfying \cref{l1app}.
	 Let $i\in\{1,\dots,m-1\}$ and suppose that for each $j\in\{1,\dots,i\}$, $e_{j}(x)\cdot f_{n}(x)\to 0$ for a.e. $x\in X$.  Then $e_{i+1}\cdot f_n\to 0$ in ${L^1(X)^m}$.
\end{lemma}
\begin{proof}
	Let $N\in\mathcal A$ be a measure zero set such that 
			\begin{align}\label{lemapp}
			e_{j}(x)\cdot f_n(x)\to 0\quad\text{for all $x\in X\setminus N$}
		\end{align}
	holds for each $j\in\{1,\dots,i\}$ and such that
	\eqref{l1app} is valid for all $x \in X \setminus N$. 
	
	Let $x\in X\setminus N$. Consider a subsequence $\{f_{n_k}(x)\}_{k\in\mathbb N}$ of $\{f_n(x)\}_{n\in\mathbb N}$ such that 
	\begin{align*}
			\liminf_{n\to\infty}e_{i+1}(x)\cdot f_{n}(x)=\lim_{k\to\infty} e_{i+1}(x)\cdot f_{n_k}(x).
	\end{align*}
	Since $F(x)$ is compact, we can find a subsequence $\{f_{n_{k_{l}}}(x)\}_{l\in\mathbb N}$ of $\{f_{n_k}(x)\}_{k\in\mathbb N}$ converging to some $f_x\in F(x)$. Then, by (\ref{lemapp}), we get $e_{j}(x)\cdot f_x=0$ for $j\in\{1,\dots, i\}$ and thus $f_x\in H_{i+1}(x)$. Since $-e_{i+1}(x)\in N_{F(x)\cap H_{i+1}(x)}(0)$, we get $e_{i+1}(x)\cdot f_x\ge 0$. Then, 
	\begin{align*}
		\liminf_{n\to\infty}e_{i+1}(x)\cdot f_{n}(x)=\lim_{l\to\infty} e_{i+1}(x)\cdot f_{n_{k_l}}(x)=e_{i+1}(x)\cdot f_x\ge 0.
	\end{align*}
	Since $x\in X\setminus N$ was arbitrary, we conclude  $\liminf_{n\to\infty}e_{i+1}\cdot f_{n}\ge0$ a.e. $x\in X$.
	It follows then from Lemma \ref{lem:weak_gives_strong_nonneg} that $|e_{i+1}\cdot f_n|_{{L^1(X)^m}}\to 0$.
\end{proof}

We are now ready to prove the main result of this subsection of the Appendix. 

\begin{theorem}\label{Visintinthm}
	Let $(X,\mathcal A,\mu)$ be a measure space, $F:X\twoheadrightarrow\mathbb R^m$  a measurable set-valued mapping taking nonempty compact convex values, and $f\in {L^1(X)^m}$ such that $f(x)\in\ext F(x)$ for a.e. $x\in X$. Let  $\{f_n\}_{n\in\mathbb N}\subset {L^1(X)^m}$ be a sequence of functions such that $f_n(x)\in F(x)$ for a.e. $x\in X$. If $f_n\rightharpoonup f$  weakly in ${L^1(X)^m}$, then $f_n\to f$ in ${L^1(X)^m}$.
\end{theorem}
\begin{proof}
	Assume without loss of generality that $f=0$.
	Let $e_1,\dots, e_m:X\to\mathbb R^m$ be the measurable functions given in \cref{norconlem}.  We argue by induction that $|e_{i}\cdot f_n|_{{L^1(X)}}\to 0$ for all $i\in\{1,\dots,m\}$. The case $i=1$ follows trivially from Lemma \ref{lem:weak_gives_strong_nonneg}.
	Let $i\in\{1,\dots,m-1\}$ and suppose that 
	\begin{align*}
		|e_{j}\cdot f_n|_{{L^1(X)^m}}\to0\quad\text{for all $j\in\{1,\dots, i\}$}.
	\end{align*}
	Let $\{e_{i+1}\cdot f_{n_k}\}_{k\in\mathbb N}$ be any subsequence of $\{e_{i+1}\cdot f_n\}_{n\in\mathbb N}$. We can find a subsequence $\{f_{n_{k_l}}\}_{l\in\mathbb N}$ of $\{f_{n_k}\}_{k\in\mathbb N}$ such that  for each  $j\in\{1,\dots, i\}$,
	\begin{align*}
		e_{j}(x)\cdot f_{n_{k_l}}(x)\to 0\quad\text{for a.e. $x\in X$}.
	\end{align*}
	It follows by \cref{preplemapp} that $\{e_{i+1}\cdot f_{n_{k_l}}\}_{l\in\mathbb N}$ converges to 0 in ${L^1(X)}$.  Since every subsequence of $\{e_{i+1}\cdot f_n\}_{n\in\mathbb N}$ has further a subsequence converging to $0$ in ${L^1(X)}$, it follows that the entire sequence converges to $0$ in ${L^1(X)}$. This completes the induction step.
	
	Finally, since $\{e_1(x),\dots, e_{m}(x)\}$ is an orthonormal basis of $\mathbb R^m$ for a.e. $x\in X$, 
	\begin{align*}
		\int_{X}|f_n(x)|\, d\mu(x)\le \sum_{i=1}^m\int_{X}|e_{i}(x)\cdot f_{n}(x)|\,d\mu(x)\longrightarrow 0.
	\end{align*}
\end{proof}

	\subsection{The weak clustering principle}

Given $p\in[1,\infty]$ and a set-valued mapping $G:X\twoheadrightarrow\mathbb R^m$ from a measure space to the Euclidean space, we denote 
\begin{align*}
	\mathcal S^p_{G}: = \{g \in L^p(X)^m : g(x) \in G(x)\text{ for a.e.\ } x \in X\}.
\end{align*}

\begin{lemma}\label{extpointlem}
	Let $(X,\mathcal A,\mu)$ be a $\sigma$-finite measure space, $F:X\twoheadrightarrow\mathbb R^m$  a measurable set-valued mapping taking nonempty compact convex values, $p\in[1,\infty]$, and $f\in L^p(X)^m$ such that $f(x)\in F(x)$ for a.e.\ $x\in X$. Suppose there exists a set $E\subset X$ of positive measure such that  
	\begin{align*}
		f(x)\notin \ext F(x)\quad \text{for a.e. $x\in E$.}
	\end{align*}
	Then there exist two distinct functions $\alpha,\beta\in \mathcal S_{F}^p$ such that $f=\displaystyle\frac{\alpha+\beta}{2}$ a.e. in $X$.
\end{lemma}

\begin{proof}
Observe that $f\in \mathcal S^p_F$, and consequently $\mathcal S^p_{F}$ is nonempty. As stated in \cite[Theorem 6.4.26]{Anhand}, from the measurability of $F$, we can deduce that $\ext\mathcal S^p_F=\mathcal S_{\ext F}^p$. Hence, as $f\not\in \mathcal S_{\ext F}^p$, $f$ is not an extreme point of $\mathcal S_F^p$; thus there must exist two distinct functions $\alpha,\beta\in \mathcal S_{F}^p$ such that $f=2^{-1}(\alpha+\beta)$ a.e. $x\in X$.
\end{proof}

\begin{lemma}\label{exissequ}
	Let $(X,\mathcal A,\mu)$ be a non-atomic $\sigma$-finite measure space and $F:X\twoheadrightarrow\mathbb R^m$  a measurable set-valued mapping taking nonempty compact convex values. Let $p\in[1,\infty)$ and consider $f\in L^p(X)^m$ such that $f(x)\in F(x)$ for a.e. $x\in X$. There exists a sequence $\{f_n\}_{n\in\mathbb N}\subset L^p(X)^m$ with the following properties.
	\begin{enumerate}
		\item $f_n(x)\in \ext F(x)$ for a.e $x\in X$ and all $n\in\mathbb N$;
		
		\item $f_n\rightharpoonup f$ weakly in $L^p(X)^m$.
	\end{enumerate}
\end{lemma}
\begin{proof} 
	Observe that $f\in \mathcal S_F^p$, and consequently $\mathcal S_{F}^p$ is nonempty. According to \cite[Proposition 6.4.27]{Anhand}, we have $\overline{\mathcal S_{\ext F}^p}^{w}=\mathcal S^p_{F}$, where $\overline{\mathcal S_{\ext F}^p}^{w}$ denotes the closure of $\mathcal S_{\ext F}^p$ with respect to the weak topology of ${L^p(X)^m}$.  By the theorem of Eberlein--\v{S}mulian (in the form of \cite[Theorem 2.8.6]{Megginson}), every sequence in $\mathcal S^p_{\ext F}$ has a weak limit point.  We can then employ Day's Lemma (\cite[Lemma 2.8.5]{Megginson}) to find a sequence $\{f_n\}_{n\in\mathbb N}\subset\mathcal S_{\text{ext F}}^p \subset {L^p(X)^m}$ such that
	$f_n  \rightharpoonup f$ in $L^p(X)^m$.
\end{proof}

We are now ready to prove the main result of this subsection.
\begin{theorem}\label{Clustering}
	Let $(X,\mathcal A,\mu)$ be a  non-atomic $\sigma$-finite measure space, $F:X\twoheadrightarrow\mathbb R^m$  a measurable set-valued mapping taking non-empty compact convex values, and $f\in {L^1(X)^m}$ such that $f(x)\in F(x)$ for a.e. $x\in X$. Suppose that there exists a set $E$ of positive measure such that 
	\begin{align*}
		f(x)\notin \ext F(x)\quad\text{for a.e. $x\in E$}.
	\end{align*}
	Then there exists $\delta_0>0$ such that for every $\delta\in(0,\delta_0]$ there exists a sequence $\{f_n\}_{n\in\mathbb N}\subset {L^1(X)^m}$ with the following properties. 
	\begin{enumerate}
		\item $f_n(x)\in F(x)$ for a.e. $x\in X$ and all $n\in\mathbb N$;
		
		\item $|f_{n}-f|_{{L^1(X)^m}}=\delta$ for all $n\in\mathbb N$; 
		
		\item $f_n\rightharpoonup f$ weakly in ${L^1(X)^m}$.
	\end{enumerate}
\end{theorem}
\begin{proof}
	By \cref{extpointlem}, there exist two distinct functions $\alpha,\beta\in {L^1(X)^m}$ such that $f(x)=2^{-1}\big(\alpha(x)+\beta(x)\big)$ for  a.e. $x\in X$. 
	Let $G:X\twoheadrightarrow\mathbb R^m$ be given by $G(x):=[\alpha(x),\beta(x)]$.
	By Lemma \ref{exissequ}, there exists a sequence $\{g_n\}_{n\in\mathbb N}\subset {L^1(X)}$ converging weakly in ${L^1(X)^m}$ to $f$ such that 
	\begin{align*}
		g_{n}(x)\in\{\alpha(x),\beta(x)\}\quad\text{for a.e $x\in X$ and all $n\in\mathbb N$}.
	\end{align*}
	Define $\delta_0:=2^{-1}|\beta-\alpha|_{{L^1(X)^m}}$. For each  $\delta\in(0,\delta_0]$ consider the sequence $\{f_n\}_{n\in\mathbb N}$ given by
	\begin{align*}
		f_n:=f+\frac{2\delta}{\hspace*{0.3cm}|\beta-\alpha|_{{L^1(X)^m}}}\big(g_n-f\big)\quad\forall n\in\mathbb N.
	\end{align*}
	Then, by construction, $f_n\rightharpoonup f$ and $|f_n-f|_{{L^1(X)^m}}=\delta$ for all $n\in\mathbb N$. The sequence $\{f_n\}_{n\in\mathbb N}$ satisfies all the stated properties and the proof culminates.
\end{proof}

Observe that the  non-atomicity assumption in the previous proposition is needed as spaces like $l^1(\mathbb N)$ have the \text{Schur's property}. Not to mention the $L^1$-spaces induced by the counting measure over a finite subset of the natural numbers, yielding finite dimensional spaces.

Finally,
we note that \cref{Clustering}
yields the result of \cite[Proposition~1]{Visin}
under a set of slightly different assumptions.


\bibliographystyle{siamplain}
\bibliography{references}
\end{document}